\documentclass{birkjour}
\usepackage{tikz}
\usepackage{amssymb,color}
\usepackage{epsfig,enumerate, amscd, amsmath,amsfonts,amssymb}
\usepackage{enumitem,graphicx}
\newcommand{\ld} {\langle}
\newcommand{\rd} {\rangle}
\newcommand{\mc} {\mathcal}

\newcommand{\la} {\lambda}

\newcommand{\noi} {\noindent}

\newcommand{\mb}{\mathbb}

\newcommand{\ds} {\displaystyle}

\newtheorem{thm}{Theorem}[section]
 \newtheorem{cor}[thm]{Corollary}
 \newtheorem{lem}[thm]{Lemma}
 \newtheorem{prop}[thm]{Proposition}
 \theoremstyle{definition}
 \newtheorem{defi}[thm]{Definition}
 \theoremstyle{remark}
 \newtheorem{rem}[thm]{Remark}
 
 \numberwithin{equation}{section}
 \catcode`\@=11
\def\theequation{\@arabic{\c@section}.\@arabic{\c@equation}}
\catcode`\@=12

\begin{document}
\title[Fractional Kirchhoff system]{Nehari manifold for  fractional Kirchhoff\\ system with critical nonlinearity}
\author{J. M. do \'O}
\address{Department of Mathematics,\\
Bras\'{\i}lia University,\\ 70910-900, Bras\'ilia, DF, Brazil}
\email{jmbo@pq.cnpq.br}
\author[J. Giacomoni]{J. Giacomoni}
\address{
LMAP (UMR {E2S-UPPA} CNRS 5142),\\ Bat. IPRA, Avenue de l'Universit\'e,\\ F-64013 Pau, France}
\email{jacques.giacomoni@univ-pau.fr}
\author{P. K. Mishra}
\address{Department of Mathematics,\\ Federal University of Para\'{\i}ba,\\ 58051-900, Jo\~ao Pessoa-PB, Brazil}
\email{pawanmishra31284@gmail.com}
\thanks{Research supported in part by INCTmat/MCT/Brazil, CNPq and CAPES/Brazil.}
\subjclass{Primary 35J47; Secondary 35J20, 47G20}

\keywords{Nehari manifold, fractional Kirchhoff system, critical exponent, multiplicity}

\begin{abstract}
\noindent In this paper, we show the existence and multiplicity of positive solutions of the  following fractional Kirchhoff system\\
\begin{equation*}
\left\{
\begin{array}{rllll}
\mc L_M(u)&=\lambda f(x)|u|^{q-2}u+ \frac{2\alpha}{\alpha+\beta}\left|u\right|^{\alpha-2}u|v|^\beta &  \text{in } \Omega,\\
\mc L_M(v)&=\mu g(x)|v|^{q-2}v+ \frac{2\beta}{\alpha+\beta}\left|u\right|^{\alpha}|v|^{\beta-2}v & \text{in } \Omega,\\
u&=v=0 &\mbox{in } \mathbb{R}^{N}\setminus \Omega,
\end{array}
\right.
\end{equation*}
where $\mc L_M(u)=M\left(\ds \int_\Omega|(-\Delta)^{\frac{s}{2}}u|^2dx\right)(-\Delta)^{s} u $
is a double non-local operator due to  Kirchhoff term $M(t)=a+b t$ with $a, b>0$ and fractional Laplacian $(-\Delta)^{s}, s\in(0, 1)$. We consider that
$\Omega$ is a bounded domain in $\mathbb{R}^N$, {$2s<N\leq 4s$} with smooth boundary,  $f, g$ are sign changing continuous functions, $\lambda, \mu>0$ are {real} parameters, $1<q<2$, $\alpha, \beta\ge 2$ {and}  $\alpha+\beta=2_s^*={2N}/(N-2s)$ {is a fractional critical exponent}. Using the idea of Nehari manifold technique and a compactness result based on {classical idea of Brezis-Lieb Lemma},  we prove the existence of at least two positive solutions for  $(\lambda, \mu)$ lying in a suitable subset of $\mathbb R^2_+$.\\
\end{abstract}

\maketitle
\section{Introduction}
\noindent In this paper, we show the existence and multiplicity of positive solutions for the  following fractional Kirchhoff system\\
\begin{equation*}
(P_{\lambda,\mu})\;\left\{
\begin{array}{rllll}
\mc L_M(u)&=\lambda f(x)|u|^{q-2}u+ \frac{2\alpha}{\alpha+\beta}\left|u\right|^{\alpha-2}u|v|^\beta& \text{in } \Omega,\\
\mc L_M(v)&=\mu g(x)|v|^{q-2}v+ \frac{2\beta}{\alpha+\beta}\left|u\right|^{\alpha}|v|^{\beta-2}v& \text{in } \Omega,\\
u&=v=0 &\mbox{in } \mathbb{R}^{N}\setminus \Omega,
\end{array}
\right.
\end{equation*}
where $\mc L_M(u)=M\left(\ds \int_\Omega|(-\Delta)^{\frac{s}{2}}u|^2dx\right)(-\Delta)^{s} u $
is a double non-local operator emerged from the fusion of  Kirchhoff term $M(t)=a+b t$ with $a, b>0$ and fractional Laplacian $(-\Delta)^{s}, s\in(0, 1)$ (see \cite{MR3120682} for modeling of a vibrating string which gives rise to this kind of operators). We consider that
$\Omega$ is a bounded domain in $\mathbb{R}^N$, {$2s<N\leq 4s$} with smooth boundary,  $\lambda, \mu>0$ are {real} parameters, $1<q<2$, $\alpha, \beta\ge 2$ {and}  $\alpha+\beta=2_s^*={2N}/(N-2s)$ {is a fractional critical exponent}. \\
Define $f^+(x)=\max\{f(x), 0\}$ and  $g^+(x)=\max\{g(x), 0\}$. 
Now for the ease of reference, we impose the following condition on the continuous weights  $f, g$:
\begin{equation*}
\textbf{(fg)}\;\;
\begin{aligned}
f, g \in L^{\gamma_*}(\Omega), \text{ where }\gamma_*>\gamma={2^*_s}/{(2^*_s-q)} \text{ and } &f^+, g^+\not\equiv 0.
\end{aligned}
\end{equation*}

\noindent In this work, we prove the multiplicity of positive solutions for the system $(P_{\lambda, \mu})$ using a popular technique of Nehari manifold. We show the multiplicity result by extracting Palais-Smale sequences in the non-empty decompositions of Nehari manifold  for $(\lambda,\mu)$ lying in some suitable subset of $\mathbb R^2_+$.  In the presence of critical nonlinearity (the lack of compactness of the critical Sobolev embedding), the Kirchhoff term  may not allow the weak limit of Palais-smale sequence to be a weak solution of the problem. This makes this class of problems interesting to study.

\noindent To overcome the nonlocal behaviour caused by nonlocal fractional operator, we have adopted the idea of Caffarelli-Silvestre extension, see \cite{MR2354493} but in the process we lack the explicit form of extremal functions. We have used the suitable asymptotic estimates of the extremals, obtained in \cite{MR2911424, MR3117361}. Apart from these challenges, due to critical growth,  we  lack the  compactness which is needed in case of  $M\not \equiv1$ to get the pre-compactness of Palais-Smale sequences. Based on the classical idea of Brezis-Lieb Lemma, we have derived a compactness result (see Proposition \ref{crcmp} in section 4) to prove Theorem \ref{22mht1}.

\noindent For $u=v,\; \alpha=\beta,\; \lambda=\mu$ and $f=g$, the problem $(P_{\lambda,\mu})$ reduces into the scalar case.
 We cite \cite{PDH}, for related results in case of scalar equation, where authors have applied a classical idea of concentration-compactness Lemma due to P. L. Lions. In the local setting and in the absence of Kirchhoff term, the systems of linear and quasilinear elliptic equations have been studied in   \cite{MR2213895,MR1776922,MR1970963,MR2388846} and references therein.

\noindent In the case, $M\equiv 1$ the fractional systems have been studied in \cite{MR3400517,MR3503657}. Precisely, for $ f=g=1 $, in \cite {MR3503657}, authors have studied the following system with critical growth
\begin{align*}
(-\Delta)^{s}u
&=\lambda |u|^{q-2}u+ \frac{2\alpha}{\alpha+\beta}\left|u\right|^{\alpha-2}u|v|^\beta\;  \text{in } \Omega,\\
(-\Delta)^{s}v &=\mu |v|^{q-2}v+ \frac{2\beta}{\alpha+\beta}\left|u\right|^{\alpha}|v|^{\beta-2}v\;  \text{in } \Omega,\\
u&=v=0 \;\mbox{in } \mathbb R^N\setminus \Omega,
\end{align*}
 where $\alpha, \beta>1$ and $\alpha+\beta=2^*_s$, $\Omega\subset \mathbb R^N$ is open bounded domain with smooth boundary, $N>2s, 0<s<1$, $1<q<2$
and proved the multiplicity of solutions using the idea of Nehari manifold and harmonic extension for suitable choice of $\lambda, \mu>0$. The extension of these results for quasilinear problems with critical growth can be seen in \cite{MR3562945}. \\
\noi However, in case of Kirchhoff fractional systems, in \cite{MR3400517, BR}, authors have obtained the multiplicity of solutions for $p$-fractional Kirchhoff system in the subcritical case for any $p\geq 2$. To the best of our knowledge, there is no results for the multiplicity of fractional Kirchhoff systems with critical nonlinearity. \\
\noindent Define the set
\begin{equation*}
 \mc M_{_\Gamma}=\left\{(\lambda,\mu)\in \mathbb R^2_+:\left((\lambda\|f\|_\gamma)^\frac{2}{2-q}+(\mu\|g\|_\gamma)^\frac{2}{2-q}\right)^\frac{2-q}{2}<\Gamma\right\}.
\end{equation*}
In the case of critical nonlinearity, we prove the following theorem.
\begin{thm}\label{22mht1}
Assume $a>0$ and {$b>0$}. Then
 \begin{enumerate}[label=(\roman*)]
\item  there exists a $\Gamma_{0}>0$ such that problem $(P_{\lambda,\mu})$ has at least one positive solution for $(\lambda,\mu) \in \mc M_{_{\Gamma_{0}}}$ with negative energy.
\item {Assume in addition that $\{x\in\Omega\,|\, f(x)>0 \mbox{ and } g(x)>0\}$ has a non zero measure. Then, for $b\in (0, b_0)$ for some $b_0>0$,  there exists $0<\Gamma_{00}=\Gamma_{00}(b) <\Gamma_{0}$} such that for $(\lambda,\mu) \in \mc M_{_{\Gamma_{00}}}$ the problem $(P_{\lambda,\mu})$ has at least two positive solutions and among them one has positive energy.
\end{enumerate}
\end{thm}
\noindent Theorem \ref{22mht1} compliments the result of \cite{MR3400517}  for critical case and \cite{BR} for non-degenerate Kirchhoff case with critical nonlinearity in low dimensions. By overcoming the lack of compactness,  we have obtained two nontrivial positive solutions by avoiding the possibility  of semi trivial solutions. 
\begin{rem}
We remark that the restriction on $b$ in Theorem \ref{22mht1} $(ii)$ is required while showing the Palais-Smale  level  in a suitable Nehari subset lies within the compactness level (see Lemma \ref{II}). 
\end{rem}
\begin{rem}
We remark that the  multiplicity result of Theorem \ref{22mht1} $(ii)$ holds true for all $b>0$ when  $(N, s)$ pair takes the following values.
\[
\left(1, \frac14\right),\;\left(2, \frac12\right), \;\left(3, \frac34\right), \;\left(1, \frac13\right), \;\left(2, \frac23\right).
\]
Observe that, in these cases either $2^*_s=4$ or $2^*_s=6$ which  allows us to find an explicit  positive root of an algebraic polynomial of degree $2^*_s$ involved in the estimation of energy level in Nehari manifold (see Lemma \ref{II}). Since $2^*_s$ is not an integer, in general, it is not always possible to find a zero of an algebraic expression.
\end{rem}
\begin{rem} The restriction $2s<N\leq 4s$ (or $2^*_s\geq 4$) can be removed if we consider a general Kirchhoff model as 
\begin{equation*}
M(t)=a+bt^{\theta-1},\quad\mbox{with }a>0 \mbox{ and } b>0,\mbox{ for any }t\geq0,\mbox{ and } \theta\in[1,2^*_s/2)
\end{equation*}
which includes the standard Kirchhoff case (for $\theta=2$) as well. For the ease of presentation of the technique we have considered only the standard Kirchhoff case.
\end{rem}
\noi The paper is organized as follows. In section 2, preliminaries and variational settings are defined. In section 3, we introduce Nehari manifold and analysis of fibering maps. In section 4, we study the compactness of Palais-Smale sequences extracted from Nehari decompositions. In section 5, the existence of first solution is proved. Finally, in section 6 we have shown the existence of second solution and concluded the proof of Theorem \ref{22mht1}. A proof of an elementary inequality is added in the Appendix for the sake of completeness. 

\section{Variational Formulation}
\noi The fractional powers of Laplacian, $(-\Delta)^{s}$, in a bounded domain $\Omega$ with zero Dirichlet boundary data are defined through the spectral decomposition using the powers of the eigenvalues of the
Laplacian operator. Let $(\varphi_j,\rho_j)$ be the eigen-functions and
eigen-vectors of $(-\Delta)$ in $\Omega$ with zero Dirichlet boundary data. Then
$(\varphi_j,\rho_j^{s})$ are the eigen-functions and eigen-vectors of
$(-\Delta)^{s}$ with  Dirichlet boundary conditions. In fact, the
fractional Laplacian $(-\Delta)^{s}$ is well defined in the space of
functions
$$
H^{s}_0(\Omega)=\left\{u=\sum c_j
  \varphi_j\in L^2(\Omega)\  :\  \|u\|_{H^{s}_0(\Omega)}=
  \left(\sum c_j^2\rho_j^{s}\right)^{1/2}<\infty\right\}
$$
and as a consequence,
$$
(-\Delta)^{s}u=\sum
  c_j\rho_j^{s}\varphi_j\,.
$$
Note that
$\|u\|_{H^{s}_0(\Omega)}=\|(-\Delta)^{s/2}u\|_{L^2(\Omega)}$. The dual space $H^{-s}(\Omega)$ and the inverse operator $(-\Delta)^{-s}$ are  defined in the standard way. The Energy functional associated to  $(P_{\lambda,\mu})$ is given by
\begin{align*}
\mathcal{J}_{\lambda,\mu}(u,v)&=\frac{1}{2}\widehat M\left(\int_\Omega|(-\Delta)^{\frac{s}{2}}u|^2dx\right)+\frac{1}{2}\widehat M\left(\int_\Omega|(-\Delta)^{\frac{s}{2}}v|^2dx\right)\\&\quad-\frac{\lambda}{q}\int_\Omega f(x)|u|^qdx+\frac{\mu}{q}\int_\Omega g(x)|v|^qdx-\frac{2}{\alpha+\beta}\int_\Omega |u|^\alpha|v|^\beta dx,
\end{align*}
where $\widehat{M}(t)=\int_0^t M(s) ds $ is the primitive of $M$. Under $\textbf{(fg)}$, it is easy to see that the functional $\mathcal{J}_{\lambda,\mu}$ is well defined and continuously differentiable on $ H_0^{s}(\Omega)\times H_0^{s}(\Omega)$.
\begin{defi}
A function $(u,v)\in H_0^{s}(\Omega)\times H_0^{s}(\Omega)$ is called  a weak solution of $(P_{\lambda,\mu})$ if for all $(\varphi,\psi) \in H_0^{s}(\Omega)\times H_0^{s}(\Omega)$
\begin{align*}
&M\left(\int_\Omega|(-\Delta)^{\frac{s}{2}}u|^2dx\right)\displaystyle \int_\Omega (-\Delta)^\frac{s}{2} u (-\Delta)^\frac{s}{2}\varphi dx+\\
&M\left(\int_\Omega|(-\Delta)^{\frac{s}{2}}v|^2dx\right)\displaystyle \int_\Omega (-\Delta)^\frac{{s}}{2} v (-\Delta)^\frac{{s}}{2}\psi dx= \displaystyle\lambda\int_\Omega f(x)|u|^{q-2}u\varphi\,dx\\&+\displaystyle \mu\int_\Omega g(x)|v|^{q-2}v\psi\,dx+\frac{2\alpha}{\alpha+\beta}\int_\Omega \left|u\right|^{\alpha-2}u|v|^\beta\varphi dx\\&+\frac{2\beta}{\alpha+\beta}\int_\Omega \left|u\right|^{\alpha}|v|^{\beta-2}v\psi dx.
\label{wf}
\end{align*}
\end{defi}
\noindent Now we discuss about a harmonic extension  technique developed by Caffarelli and Silvestre \cite{MR2354493} to treat the non-local problems involving fractional Laplacian. In this technique, we study an extension problem corresponding to a nonlocal problem so that we can investigate the non-local problem via classical variational methods. To begin with, we first define the harmonic extension of $u\in H^s_0(\Omega)$ as follows
 \begin{defi} \label{HED}
 For $u\in H^s_0(\Omega)$, the harmonic extension $E_{s}(u):=w$ is the solution of the following problem
 \begin{equation*}
\left\{\begin{array}{rlll}
-\mathrm{div}(y^{1-2s}\nabla w) &=0 \quad \text{in}\quad  \Omega\times (0, \infty):=\mathcal C,\\
w&=0 \quad \text{on}\quad \partial \Omega\times (0, \infty):=\partial_L,\\
w&=u \quad \text{on}\quad \Omega \times\{0\}.
\end{array}
\right.
\end{equation*}
Moreover, the extension function is related with fractional Laplacian by
\[
(-\Delta)^s u(x)=-\kappa_s\displaystyle \lim_{y\rightarrow 0^+}y^{1-2s}\frac{\partial w}{\partial y}(x,y),
\]
{where $\kappa_s =2^{2s-1}\frac{\Gamma(s)}{\Gamma(1-s)}$.}
 \end{defi}
 \noi The solution space for the extension problem is the following Hilbert space
$$
E_{0}^1(\mathcal{C})=\left\{w\in
L^{2}(\mathcal{C})\,:\: w=0 \mbox{ on }
\partial_L,\;  \|w\|<\infty\right\}.
$$
The norm  $$\|w\|:=
\left(\kappa_s\int_{\mathcal{C}} y^{1-2s}|\nabla
w|^2 dx dy\right)^{1/2}
$$
is induced on $E_{0}^1(\mathcal{C})$
through the following inner product
\[
\langle w, z\rangle=\kappa_s \ds \int_{\mc C}y^{1-2s}\nabla w. \nabla z dx\,dy.
\]
 We observe that the extension operator is an isometry
between $H_0^{s}(\Omega)$ and
$E_{0}^1(\mathcal{C})$. That is, for all $u\in H_{0}^{s}(\Omega)$, we have
\begin{equation}\label{equivalnorm}
\|E_s(u)\|=\|u\|_{H^{s}_{0}(\Omega)}.
\end{equation}
This isometry in \eqref{equivalnorm} is the key to study the Kirchhoff type problems in the harmonic extension set up. Denote $E_s(u)=w$ and $E_s(v)=z$ in the sense of Definition \ref{HED}. Then, the problem $(P_{\lambda, \mu})$ is equivalent to the study of the following extension problem
  \begin{equation*}
  (S_{\lambda, \mu})\left\{\begin{array}{rll}
 -\mathrm{div} (y^{1-2s}\nabla w)&=0,\;-\mathrm{div} (y^{1-2s}\nabla z)=0, &\textrm{in}\; \mathcal{C},\\
 w=z&=0,  &\text{on}\;\partial_L,\\
  M(\|w\|^2)\frac{\partial w}{\partial \nu}&= \lambda f|w|^{q-2}w+\frac{2\alpha}{\alpha+\beta}\left|w\right|^{\alpha-2}w|z|^\beta, &\textrm{on}\;\Omega\times \{0\},\\
  M(\|z\|^2)\frac{\partial z}{\partial \nu}&= \mu g|z|^{q-2}z+\frac{2\beta}{\alpha+\beta}\left|w\right|^{\alpha}|z|^{\beta-2}z, &\textrm{on}\;\Omega\times \{0\},
\end{array}
\right.
\end{equation*}
where $\frac{\partial w}{\partial \nu}=-\kappa_s\displaystyle\lim_{y\rightarrow 0^+}y^{1-2s}\frac{\partial w}{\partial y}(x, y)$ and $\frac{\partial z}{\partial \nu}=-\kappa_s\displaystyle\lim_{y\rightarrow 0^+}y^{1-2s}\frac{\partial z}{\partial y}(x, y)$.
\noindent The natural space to look for the solution of $(S_{\lambda, \mu})$ is the product space $ \mc H(\mc C):=E^1_{0}(\mc C)\times E^1_{0}(\mc C)$ endowed with the following product norm
\[
\|(w, z)\|^2=\kappa_s\left(\int_{\mathcal{C}} y^{1-2s}|\nabla w|^2dxdy+\int_{\mathcal{C}} y^{1-2s}|\nabla z|^2dxdy\right).
\]

\noindent The variational functional $\mathcal I_{\lambda, \mu}:  \mc H(\mc C)\rightarrow \mathbb R$ associated to $(S_{\lambda, \mu})$ is defined as
\begin{align*}\label{fel}
\mathcal I_{\lambda, \mu}(w, z)&=\frac{1}{2}\widehat M(\|w\|^2)+\frac{1}{2}\widehat M(\|z\|^2) -\frac{\lambda}{q}\int_\Omega f(x)|w(x, 0)|^qdx\\&-\frac{\mu}{q}\int_\Omega g(x)|z(x, 0)|^qdx-\frac{2}{\alpha+\beta}\int_\Omega |w(x, 0)|^\alpha|z(x, 0)|^\beta dx.
\end{align*}
Without putting great efforts, it can be shown that $\mathcal I_{\lambda, \mu}$ is well defined and $C^1$. Now we give the definition of a weak solution of the extension problem $(S_{\lambda, \mu})$.
\begin{defi}
A function $(w,z)\in \mc H(\mc C)$ is called  a weak solution of  $(S_{\lambda, \mu})$ if for all $(\varphi,\psi) \in \mc H(\mc C)$
\begin{align*}
&M\left(\|w\|^2\right)\kappa_s \int_\mathcal C y^{1-2s}\nabla w.\nabla \varphi dzdy+M\left(\|z\|^2\right)\kappa_s \int_\mathcal C y^{1-2s}\nabla z.\nabla \psi dzdy-\\
& \displaystyle\lambda\int_\Omega f(x)|w(x, 0)|^{q-2}w(x, 0)\varphi(x, 0)\,dx-\displaystyle \mu\int_\Omega g(x)|z(x, 0)|^{q-2}z(x, 0)\psi(x,0)\,dx\\&\quad-\frac{2\alpha}{\alpha+\beta}\int_\Omega \left|w(x, 0)\right|^{\alpha-2}w(x, 0)|z(x, 0)|^\beta\varphi(x, 0) dx\\&\quad\frac{2\beta}{\alpha+\beta}\int_\Omega \left|w(x, 0)\right|^{\alpha}|z(x, 0)|^{\beta-2}z(x, 0)\psi(x, 0) dx=0.
\end{align*}
\end{defi}
\noi It is clear that the critical points of $\mathcal I_{\lambda, \mu}$ in $\mc H(\mc C)$ corresponds to the critical points of $\mc J_{\lambda, \mu}$ in $H^{{s}}_0(\Omega)\times H^{{s}}_0(\Omega)$. Thus if $(w, z)$ solves $ (S_{\lambda, \mu}),$ then $(u, v)$, where $u=\textrm{trace}\;(w)=w(x,0)$ , $v=\textrm{trace}\;(z)=z(x,0)$ is a solution of  $(P_{\lambda, \mu})$ and vice-verse.\\
\noindent Now, we state the following trace inequality which  will be  used in the subsequent Lemmas.
\begin{lem}\label{22traceemb}
Let $2\leq r\leq 2^*_s$, then there exists $C_r>0$ such that for all $w\in E^1_{0}(\mathcal C)$,
\begin{equation}\label{traceequation}
\displaystyle \int_{\mathcal C}y^{1-2s}|\nabla w|^2 dxdy\geq C_r\left(\int_{\Omega}|w(x, 0)|^r dx\right)^{\frac{2}{r}}.
\end{equation}
\noindent Moreover for $r=2^*_s$, the best constant in  \eqref{traceequation} will be denoted by $S(s, N)$, that is,
 \begin{equation}\label{soblc}
S(s, N):=\inf_{w\in
E^1_{0}(\mathcal {C})\backslash\{0\}}\frac{\displaystyle\int_{\mathcal
{C}}y^{1-2s}|\nabla w|^2dxdy}
{\left(\displaystyle\int_{\Omega}|w(x, 0)|^{2^*_s}dx\right)^{\frac{2}{2^*_s}}}.
\end{equation}
 and  it is indeed achieved in the case when $\Omega=\mathbb{R}^N$ and
$w=E_s(u)$, where
\begin{equation}\label{extrml}
u(x)=u_{\varepsilon}(x)= \frac{\varepsilon^{(N-2s)\slash
2}}{(|x|^2+\varepsilon^2)^{(N-2s)\slash
    2}}
\end{equation}
with $\varepsilon>0$ arbitrary.
\end{lem}

\noi We conclude this section by introducing the following minimization problem. 
\begin{equation}\label{consys}
S(s, \alpha,\beta):=\inf_{(w,z)\in
\mc H(\mathcal {C})\backslash\{0\}}\frac{\displaystyle\int_{\mathcal
{C}}y^{1-2s}(|\nabla w|^2+|\nabla z|^2)dxdy}
{\Big(\displaystyle\int_{\Omega}|w(x, 0)|^{\alpha}|z(x, 0)|^{\beta}dx\Big)^{\frac{2}{2^*_s}}}.
\end{equation}
In light of the inequality $|w|^\alpha|z|^\beta\leq|w|^{\alpha+\beta}+|z|^{\alpha+\beta}$ and Lemma \ref{22traceemb}, the best Sobolev constant in \eqref{consys} is well defined. Using the ideas from  \cite{MR1776922}, authors in \cite{MR3503657} established the following  relationship between
$S(s,N)$ and $S(s, \alpha,\beta)$.

\begin{lem}\label{optsys}
For the constants $S(s,N)$ and $S(s, \alpha,\beta)$ introduced in  \eqref{soblc} and \eqref{consys} respectively, it holds
\begin{equation*}
S(s, \alpha,\beta)=
\left[\left(\frac{\alpha}{\beta}\right)^{\frac{\beta}{\alpha+\beta}}+\left(\frac{\beta}{\alpha}\right)^{\frac{\alpha}{\alpha+\beta}}
\right]S(s, N).
\end{equation*}
In particular, the constant $S(s,\alpha,\beta)$ is achieved for $\Omega=\mb R^N$.
\end{lem}

\section{Nehari manifold for $(S_{\lambda,\mu})$}
\noindent In this section we study the nature of Nehari manifold associated with $(S_{\lambda,\mu})$.  In the case $\alpha+\beta \ge 4$, the functional $\mathcal{I}_{\lambda,\mu}$ is not bounded below on $\mc H(\mc C)$. We will show that it is bounded on some suitable subset of $\mc H(\mc C)$ and on minimizing $\mathcal{I}_{\lambda,\mu}$ on these subsets, we get the solutions of $(S_{\lambda,\mu})$.
\noindent We define the Nehari set $\mathcal N_{\lambda,\mu}$  as
\begin{equation*}
\mathcal  N_{\lambda,\mu}=\{(w,z)\in \mc H(\mc C)\setminus\{0\}: \left\langle \mathcal I_{\lambda,\mu}^\prime(w,z), (w,z)\right\rangle_* =0\},\end{equation*}
where $\langle,\;\rangle_*$ is the duality between $\mc H(\mc C)$ and its dual space. Thus $(w,z)\in \mathcal N_{\lambda,\mu}$ if and only if
\begin{equation}\label{22nlam}
\begin{aligned}
&M(\|w\|^2)\|w\|^2+M(\|z\|^2)\|z\|^2=\lambda \int_{\Omega} f(x)|w(x, 0)|^q dx\\&+\mu \int_{\Omega} g(x)|z(x, 0)|^q dx+
2\int_{\Omega}|w(x, 0)|^{\alpha}|z(x, 0)|^\beta dx.
\end{aligned}
\end{equation}
Now for fix $(w,z) \in \mc H(\mc C)$, we define $\phi_{(w,z)}: \mathbb R^+\rightarrow \mathbb R$, much known as fiber maps,  as $\phi_{(w,z)}(t)=\mathcal{I}_{\lambda,\mu}(tw,tz)$.
Thus
\begin{align*}
\phi_{(w,z)}(t)&=\frac{1}{2}\widehat{M}(t^2\|w\|^2)+\frac{1}{2}\widehat{M}(t^2\|z\|^2)-\frac{t^q}{q}\lambda\int_{\Omega} f(x)|w(x, 0)|^qdx\\ 
&-\frac{t^q}{q}\mu\int_{\Omega} g(x)|z(x, 0)|^qdx-\frac{2t^{2^*_s}}{2^*_s}\int_{\Omega} |w(x, 0)|^\alpha|z(x, 0)|^\beta dx.
\end{align*}
Now, for fixed $(w, z)\in \mc H(\mc C)$, taking derivative with respect to the variable $t$ and putting $t=1$, we get
\begin{equation}\label{phid1}
\begin{aligned}
\phi_{(w,z)}^{\prime}(1)&=M(\|w\|^2)\|w\|^2+M(\|z\|^2)\|z\|^2
-\lambda \int_{\Omega} f(x)|w(x, 0)|^qdx\\
&\quad-\mu \int_{\Omega} g(x)|z(x, 0)|^qdx-2\int_{\Omega}|w(x, 0)|^\alpha|z(x, 0)|^\beta dx,
\end{aligned}
\end{equation}
\begin{equation}\label{phid2}
\begin{aligned}
\phi_{(w,z)}^{\prime\prime}(1)&= a(\|w\|^2+\|z\|^2)+ 3b(\|w\|^{4}+\|z\|^{4})
-(q-1)\lambda\int_{\Omega} f(x)|w(x, 0)|^qdx\\&
- (q-1)\mu\int_{\Omega} g(x)|z(x, 0)|^qdx -2(2^*_s-1)\int_{\Omega} |w(x, 0)|^\alpha|z(x, 0)|^\beta dx.
\end{aligned}
\end{equation}
From equation \eqref{phid1}, $(w,z)\in \mathcal N_{\lambda,\mu}$ if and only if
$\phi_{(w,z)}^{\prime}(1)=0$. Thus it is natural to split
$\mathcal {N}_{\lambda,\mu}$ into three parts corresponding to local minima,
local maxima and points of inflection. For this, we set
\begin{align*}
\mathcal N_{\lambda,\mu}^{\pm} &:= \left\{(w,z)\in \mathcal N_{\lambda,\mu}: \phi_{(w,z)}^{\prime\prime}(1) \gtrless0\right\},\\
\mathcal N_{\lambda,\mu}^{0}&:= \left\{(w,z)\in \mathcal N_{\lambda,\mu}: \phi_{(w,z)}^{\prime\prime}(1) = 0\right\}.
\end{align*}
\noi The following Lemma shows that minimizers for $\mc I_{\la, \mu}$ on $\mc N_{\la, \mu}$ are critical points for $\mc I_{\la, \mu}$.
\begin{lem}\label{lmcp}
If $(w,z)$ is a minimizer of $\mathcal{I}_{\lambda,\mu}$ on $\mathcal{N}_{\lambda,\mu}$ such that $(w,z) \notin \mathcal{N}_{\lambda,\mu}^{0}.$ Then $(w, z)$ is a critical point for $\mathcal{I}_{\lambda,\mu}.$
\end{lem}
\begin{proof}
Let $(w, z)$ be a local minimizer for $\mc I_{\lambda, \mu}$ in any of the
subsets of $\mc N_{\la, \mu}$. Then, in any case $(w, z)$ is a
minimizer for $\mc I_{\lambda, \mu}$ under the constraint $\mc F_{\la, \mu}(w, z):=\left\langle
\mc I_{\lambda, \mu}^{\prime}(w, z),(w, z)\right \rangle_* =0$. {Since $(w,z)\not\in\mathcal{N}_{\lambda,\mu}^{0}$, the constraint is non degenerate in $(w,z)$, then by the Lagrange
multipliers rule}, there exists $\eta \in \mb R$ such that $ \mc I_{\lambda, \mu}^{\prime}(w, z)= \eta \mc F_{\la, \mu}^{\prime}(w, z)$. Thus $\ld
\mc I_{\lambda, \mu}^{\prime}(w, z), (w, z)\rd= \eta \;\left\langle \mc F_{\lambda, \mu}^{\prime}(w, z), (w, z)\right\rangle_* = \eta
\phi_{(w, z)}^{\prime\prime}(1)$=0, but $(w, z)\notin \mc N_{\la, \mu}^{0}$ and so
$\phi_{(w,z)}^{\prime\prime}(1) \ne 0$. Hence $\eta=0$ which completes the
proof of the Lemma.
\end{proof}
\noindent Denote

\[\Gamma_1:=\left(\frac{a(2^*_s-2)(\kappa_s S(s, N))^\frac{q}{2}}{2^*_s-q}\right)
    \left(\frac{a(2-q)(\kappa_s S(s, N))^\frac{2^*_s}{2}}{2(2^*_s-q)}\right)^{\frac{2-q}{2^*_s-2}}.
\]The next Lemma helps us to show that $\mathcal{N}_{\lambda,\mu}$ is a manifold for suitable choice of $(\lambda,\mu)$. 
\begin{lem}\label{l22.2}
Let $(\lambda,\mu)\in \mc M_{_{\Gamma_1}}$ then $\mathcal N_{\lambda,\mu}^0=\emptyset$.
\end{lem}
\begin{proof}
We prove this lemma by contradiction. Assume on the contrary that $(w,z)\in \mc N_{\lambda,\mu}^0$. Then we have two cases\\
\textbf{Case 1:} $(w,z)\in \mc N_{\lambda,\mu}$ and $\lambda\int_{\Omega} f(x)|w(x, 0)|^qdx+\mu \int_{\Omega} g(x)|z(x, 0)|^qdx=0$.\\
\noindent  From \eqref{phid1} and \eqref{phid2} we get,
\begin{align*}
a(\|w\|^2+\|z\|^2) &+ 3b(\|w\|^{4}+\|z\|^{4})-2(2^*_s-1)\int_{\Omega } |w(x, 0)|^\alpha|z(x, 0)|^\beta dx\\
&=(2-2^*_s)a(\|w\|^2+\|z\|^2)+ (4-2^*_s)b(\|w\|^{4}+\|z\|^{4})<0
\end{align*}
which is a contradiction.\\
\textbf{Case 2:} $(w,z)\in \mc N_{\lambda,\mu}$ and $\lambda\ds \int_{\Omega}  f(x)|w(x, 0)|^qdx+\mu \int_{\Omega }  g(x)|z(x, 0)|^qdx\neq0$.\\
\noindent  Again from \eqref{phid1} and \eqref{phid2}, we have
\begin{equation}\label{2.11}
\begin{aligned}
    (2-q)a (\|w\|^{2}+\|z\|^{2}) &+ (4-q)b(\|w\|^{4}+\|z\|^{4}) \\&= 2(2^*_s-q)\int_{\Omega }  |w(x, 0)|^\alpha|z(x, 0)|^\beta dx
    \end{aligned}
    \end{equation}
    and 
    \begin{equation}\label{2.12}
    \begin{aligned}
    (2^*_s-2)& a (\|w\|^{2}+\|z\|^{2})+ (2^*_s-4) b(\|w\|^{4}+\|z\|^{4})\\
    &= (2^*_s-q) \left(\lambda \int_{\Omega }  f(x)|w(x, 0)|^q dx+\mu\int_{\Omega }  g(x)|z(x, 0)|^q dx\right).
  \end{aligned}
  \end{equation}

 \noindent Now define $\mc T_{\lambda,\mu}: \mc N_{\lambda,\mu} \rightarrow \mathbb R$ as
  \begin{align*}
    \mc T_{\lambda,\mu}(w,z) &=\frac{(2^*_s-2) a (\|w\|^{2}+\|z\|^{2}) + (2^*_s-4) b(\|w\|^{4}+\|z\|^{4})}{(2^*_s-q)}\\&- \lambda\int_{\Omega }  f(x)|w(x, 0)|^{q} dx-\mu\int_{\Omega }  g(x)|z(x, 0)|^{q} dx
  \end{align*}
  then from \eqref{2.12}, $\mc T_{\lambda,\mu}(w,z) = 0$ for all  $(w,z)\; \in \mathcal{N}_{\lambda,\mu}^{0}.$  Also,
  \begin{align*}
  &\mc T_{\lambda,\mu}(w,z)\geq \left(\frac{2^*_s-2}{2^*_s-q}\right)a(\|w\|^2+\|z\|^2 )\\
  &\quad - \lambda \int_{\Omega}  f(x)|w(x, 0)|^q dx-\mu\int_{\Omega }  g(x)|z(x, 0)|^q dx\\
  & \geq  \left(\frac{2^*_s-2}{2^*_s-q}\right)a(\|w\|^2+\|z\|^2 ) \\&\quad
  - (\kappa_s S(s, N))^\frac{-q}{2}\left((\lambda\|f\|_\gamma)^\frac{2}{2-q}+(\mu\|g\|_\gamma)^\frac{2}{2-q}\right)^\frac{2-q}{2}
  (\|w\|^2+\|z\|^2)^\frac{q}{2},\\
  & \geq  (\|w\|^2+\|z\|^2)^\frac{q}{2}\left(\left(\frac{2^*_s-2}{2^*_s-q}\right)a(\|w\|^2+\|z\|^2)^\frac{2-q}{2}\right.\\
  &\quad\left. -(\kappa_s S(s, N)^\frac{-q}{2}\left((\lambda\|f\|_\gamma)^\frac{2}{2-q}+(\mu\|g\|_\gamma)^\frac{2}{2-q}\right)^\frac{2-q}{2}\right).\\
  \end{align*}
  Now from \eqref{2.11}, we get
  \begin{equation}\label{2.13}
  \|w\|^2+\|z\|^2 \geq \left(\frac{a(2-q)(\kappa_s S(s, N))^\frac{2^*_s}{2}}{2(2^*_s-q)}\right)^{\frac{2}{2^*_s-2}}.
  \end{equation}
  Using (\ref{2.13}), we have   
  \begin{align*}
    \mc T_{\lambda,\mu}(w,z) &\geq (\|w\|^2+\|z\|^2)^\frac{q}{2}\left(\left(\frac{a(2^*_s-2)}{2^*_s-q}\right)
    \left(\frac{a(2-q)(\kappa_s S(s, N))^\frac{2^*_s}{2}}{2(2^*_s-q)}\right)^{\frac{2-q}{2^*_s-2}} \right.\\& \quad
    \left.-(\kappa_s S(s, N))^\frac{-q}{2}\left((\lambda\|f\|_\gamma)^\frac{2}{2-q}+(\mu\|g\|_\gamma)^\frac{2}{2-q}\right)^\frac{2-q}{2}\right).
  \end{align*}
  This implies that for $(\lambda,\mu)\in \mc M_{_{\Gamma_1}}, \mc T_{\lambda,\mu}(w,z)>0,$ for all $(w,z) \in \mathcal{N}_{\lambda,\mu}^{0},$
  which is a contradiction.
\end{proof}
\noindent Let us use the following notations. $\Theta_{\lambda,\mu} := \inf\{\mathcal{I}_{\lambda,\mu}(w,z)| (w,z) \in \mathcal{N}_{\lambda,\mu}\}$ and $\Theta_{\lambda,\mu}^\pm := \inf\{\mathcal{I}_{\lambda,\mu}(w,z)| (w,z) \in \mathcal{N}_{\lambda,\mu}^\pm\}$. The next Lemma justifies the choice of studying a minimization problem on $\mc N_{\lambda, \mu}$.
\begin{lem}\label{le44}
$\mathcal{I}_{\lambda,\mu}$ is coercive and bounded below on  $\mathcal{N}_{\lambda,\mu}$.
\end{lem}
\begin{proof}
For $(w,z) \in \mathcal{N}_{\lambda,u},$  using H\"{o}lder's inequality, we have
\begin{align*}
&\mathcal{I}_{\lambda, \mu}(w, z) = \left(\frac{1}{2}-\frac{1}{2^*_s}\right)a(\|w\|^{2}+\|z\|^{2})+ \left(\frac{1}{4}-\frac{1}{2^*_s}\right)b(\|w\|^{4}+\|z\|^{4})\\
&-\left(\frac{1}{q}-\frac{1}{2^*_s}\right)\left(\lambda\int_{\Omega }  f(x)|w(x, 0)|^{q}dx+\mu \int_{\Omega }  g(x)|z(x, 0)|^{q}dx\right),\\
 &\geq \left(\frac{1}{2}-\frac{1}{2^*_s}\right)a(\|w\|^{2}+\|z\|^{2})+ \left(\frac{1}{4}-\frac{1}{2^*_s}\right)b(\|w\|^{4}+\|z\|^{4})\\
 &-\left(\frac{1}{q}-\frac{1}{2^*_s}\right)
 (\kappa_s S(s, N))^\frac{-q}{2}\left((\lambda\|f\|_\gamma)^\frac{2}{2-q}+(\mu\|g\|_\gamma)^\frac{2}{2-q}\right)^\frac{2-q}{2}
  (\|w\|^2+\|z\|^2)^\frac{q}{2}.
\end{align*}
Since  {$2^*_s\geq 4$}, $\mathcal{I}_{\lambda, \mu}$ is coercive and bounded below in $\mathcal{N}_{\lambda,\mu}$.
\end{proof}
\begin{lem}\label{3a}
Let $(\lambda,\mu)\in\mc M_{_{\Gamma_1}}$. Then ${\Theta_{\lambda,\mu}}\leq {\Theta_{\lambda,\mu}^+} <0$.
\end{lem}
\begin{proof}
Let $(w, z) \in \mathcal H(\mc C)$ be such that $$\lambda\int_{\Omega} f(x)|w(x, 0)|^{q}dx+\mu\int_{\Omega}g(x)|z(x, 0)|^qdx>0.$$ Then by Lemma \ref{L37}, for $(\lambda,\mu)\in\mc M_{_{\Gamma_1}}$,  there exists $ t^+((w, z)) > 0$ such that  $(t^+w, t^+z) \in \mathcal{N}_{\lambda,\mu}^{+}$. Denote $t^+w=w_1$ and $t^+z=z_1$. Therefore,
\begin{align*}
 \mathcal{I}_{\lambda,\mu}({w}_1,{z}_1)& =\left(\frac{1}{2}-\frac{1}{q}\right)a (\|{w}_1\|^{2} +\|z_1\|^{2})+ \left(\frac{1}{4}-\frac{1}{q}\right)b (\|{w}_1\|^{4}+\|z_1\|^{4})\\
  &+{2} \left(\frac{1}{q}-\frac{1}{2^*_s}\right) \int_{\Omega} {|{w}_1(x, 0)|^\alpha|z_1(x, 0)|^\beta dx}.
\end{align*}
Since $(w_1,z_1)\in \mathcal{N}_{\lambda,\mu}^{+}$,we get
\begin{align*}
\mathcal{I}_{\lambda,\mu}({w}_1, z_1) &\leq -\frac{(2-q)(2^*_s-2)}{2(2^*_s)q} a(\|{w}_1\|^{2} +\|z_1\|^{2})\\&\quad-\frac{(4-q)(2^*_s-4)}{4(2^*_s)q} b(\|{w}_1\|^{4} +\|z_1\|^{4})<0.
\end{align*}
\end{proof}
\noindent Let
\[\Gamma_2:= (\kappa_s S(s, N))^\frac{q}{2}{\left(\frac{S(s, \alpha, \beta)^\frac{2^*_s}{2}(2-q)}{2^*_s-2}\right)^\frac{2-q}{2^*_s-2}\left(\frac{a(2^*_s-2)}{2^*_s-q}\right)^\frac{2^*_s-q}{2^*_s-2}}.\]
 \noi Now we discuss the behavior of fibering maps with respect to the sign changing weights in the following two cases.\\
\noindent \textbf{Case 1:} $\lambda\ds \int_{\Omega} f(x)|w(x, 0)|^qdx+\mu\ds \int_{\Omega} g(x)|z(x, 0)|^qdx>0$.
First we define $\psi_{(w,z)}:\mathbb{R}^+\rightarrow \mathbb{R}$ as
\begin{align*}
\psi_{(w,z)}(t)&=at^{2-2^*_s}(\|w\|^2+\|z\|^2)+b t^{4-2^*_s}(\|w\|^{4}+\|z\|^{4})\\
&-t^{q-2^*_s}\left(\lambda\ds \int_{\Omega} f(x)|w(x, 0)|^qdx+\mu\ds \int_{\Omega} g(x)|z(x, 0)|^qdx\right).
\end{align*}
Observe that $(tw,tz)\in \mathcal{N}_{\lambda,\mu}$ if and only if
\[
\psi_{(w,z)}(t)=2\int_{\Omega} |w(x, 0)|^\alpha|z(x, 0)|^\beta dx.
\]
It is clear that 
$
\lim_{t\rightarrow 0^+}\psi_{(w,z)}(t)= -\infty $ and $ \lim_{t\rightarrow \infty}\psi_{(w,z)}(t)=0.$
Moreover,
\begin{equation}\label{siutd}
\begin{aligned}
 &\psi_{(w,z)}^{\prime}(t)=a(2-2^*_s)t^{1-2^*_s}(\|w\|^2+\|z\|^2)+b(4-2^*_s)t^{3-2^*_s}(\|w\|^{4}+\|z\|^{4})\\
 &-(q-2^*_s)t^{q-2^*_s-1}\left(\lambda\ds \int_{\Omega} f(x)|w(x, 0)|^qdx+\mu\ds \int_{\Omega} g(x)|z(x, 0)|^qdx\right).
\end{aligned}
\end{equation}
Rewrite $\psi_{(w,z)}^{\prime}(t)=t^{3-2^*_s} \mc A_{(w, z)}(t)$, where
\begin{align*}
\mc A_{(w, z)}(t)&=a(2-2^*_s)t^{-2}(\|w\|^2+\|z\|^2)+b(4-2^*_s)(\|w\|^{4}+\|z\|^{4})\\
 &-(q-2^*_s)t^{q-4}\left(\lambda\ds \int_{\Omega} f(x)|w(x, 0)|^qdx+\mu\ds \int_{\Omega} g(x)|z(x, 0)|^qdx\right).
\end{align*}
Then it can be shown that there exists unique $t^*>0$ such that $\mc A^\prime_{(w, z)}(t^*)=0$. Indeed,
$$
t^*=\left(\frac{(2^*_s-q)(4-q)\left(\lambda\ds \int_{\Omega} f(x)|w(x, 0)|^qdx+\mu\ds \int_{\Omega} g(x)|z(x, 0)|^qdx\right)}{2a(2^*_s-2)(\|w\|^2+\|z\|^2)}\right)^\frac{1}{2-q}.
$$
Also 
$$
\lim _{t\rightarrow 0^+}\mc A_{(w, z)}(t)=-\infty\;\textrm{ and }\; \lim_{t\rightarrow \infty}\mc A_{(w, z)}(t)=b(4-2^*_s)(\|w\|^{4}+\|z\|^{4}),
$$ which implies that there exists unique $t_*>0$ such that $\mc A_{(w, z)}(t_*)=0$. Therefore, from $\psi_{(w,z)}^{\prime}(t)=t^{3-2^*_s} \mc A_{(w, z)}(t)$, we get $t_*$ as a unique critical point of $\psi_{(w,z)}(t)$, which is the global maximum point.

Now we can estimate $\psi_{(w,z)}(t_{*})$ from below as follows.

\begin{align*}
\psi_{(w,z)}(t)&=\overline{\psi}_{(w, z)}(t)+b t^{4-2^*_s}(\|w\|^{4}+\|z\|^{4})\geq \overline{\psi}_{(w, z)}(t),
\end{align*}
where 
\begin{align*}
\overline{\psi}_{(w,z)}(t)&=at^{2-2^*_s}(\|w\|^2+\|z\|^2)
\\&\quad -t^{q-2^*_s}\left(\lambda\ds \int_{\Omega} f(x)|w(x, 0)|^qdx+\mu\ds \int_{\Omega} g(x)|z(x, 0)|^qdx\right).
\end{align*}
Using elementary calculus, we obtain
\begin{align*}
\max_{t>0}\overline{\psi}_{(w,z)}(t)&=\overline{\psi}_{(w,z)}(t_c)
\end{align*}
with
 $$
 t_c=\left(\frac{(2^*_s-q)\left(\lambda\ds \int_{\Omega} f(x)|w(x, 0)|^qdx+\mu\ds \int_{\Omega} g(x)|z(x, 0)|^qdx\right)}{(2^*_s-2)a(\|w\|^2+\|z\|^2)}\right)^\frac{1}{2-q}.
$$
Therefore, 
\begin{align*}
{\psi}_{(w,z)}(t_*)& \geq \overline{\psi}_{(w,z)}(t_c)\\&\geq \frac{\left(\frac{2-q}{2^*_s-2}\right)\left(\frac{a(2^*_s-2)}{2^*_s-q}\right)^\frac{2^*_s-q}{2-q}(\|w\|^2+\|z\|^2)^{\frac{2^*_s}{2}}}{\left((\kappa_s S(s, N))^\frac{-q}{2}\left((\lambda\|f\|_\gamma)^\frac{2}{2-q}+(\mu\|g\|_\gamma)^\frac{2}{2-q}\right)^\frac{2-q}{2}\right)^\frac{2^*_s-2}{2-q}}.
\end{align*}
Hence if $(\lambda,\mu)\in \mc M_{_{\Gamma_2}}$, then
there exists unique $t^+ = t^+((w,z)) < t_{*}$ and $t^- = t^-((w,z)) > t_{*},$ such that
\[\psi_{(w,z)}(t^+) = {2}\int_{\Omega}  |w(x, 0)|^{\alpha}|z(x, 0)|^\beta dx= \psi_{(w, z)}(t^-).
\]
 That is,  $(t^+w,t^+z)$ and $ (t^-w,t^-z) \in \mathcal{N}_{\lambda,\mu}.$ Besides this, $\psi^{'}_{(w,z)}(t^+) > 0$ and $\psi_{(w,z)}^{'}(t^-) < 0$ implies $(t^+w,t^+z) \in \mathcal{N}^{+}_{\lambda,\mu}$ and $(t^-w,t^-z) \in \mathcal{N}^{-}_{\lambda,\mu}.$  Since
 \[\phi^{'}_{(w,z)}(t) = t^{2^*_s}(\psi_{(w,z)}(t)- \int_{\Omega}  |w(x, 0)|^{\alpha}|z(x, 0)|^\beta dx),\]
  we get $\phi^{'}_{(w,z)}(t)<0$ for all $t \in [0, t^+)$ and $\phi^{'}_{(w,z)}(t)>0$ for all $t \in (t^+, t^-)$. Thus $\mathcal{I}_{\lambda,\mu}(t^+w,t^+z) = \displaystyle\min_{0 \leq t \leq t^-}\mathcal{I}_{\lambda,\mu} (tw,tz).$ Also $\phi^{'}_{(w,z)}(t) > 0$ for all $t \in [t^+, t^-),
\phi^{'}_{(w,z)}(t^-) = 0$ and $\phi^{'}_{(w,z)}(t) < 0$ for all $t \in (t^-, \infty)$ implies that $\mathcal{I}_{\lambda,\mu}(t^-w, t^-z)
= \displaystyle\max_{t \geq t_{*}} \mathcal{I}_{\lambda,\mu}(tw, tz).$\\
 \noindent \textbf{Case 2:} $\lambda\ds \int_{\Omega} f(x)|w(x, 0)|^qdx+\mu\ds \int_{\Omega} g(x)|z(x, 0)|^qdx \leq0$.\\
\noindent  Note that $\lim_{t \rightarrow \infty}\psi_{(w, z)}(t)=0$ and  $\lim_{t \rightarrow 0}\psi_{(w, z)}(t)=+\infty$. Additionally, $\psi_{(w, z)}(t)\geq 0$ for all $t\geq 0$ and is decreasing. Hence for all $\lambda, \mu >0$ there exists $\hat t>0$ such that $(\hat tw,\hat tz)\in \mathcal{N}_{\lambda,\mu}^-$ and $\mathcal{I}_{\lambda,\mu}(\hat tw,\hat tz)=\displaystyle\max_{ t\geq 0} \mathcal{I}_{\lambda,\mu}(tw,tz)$. 
 \begin{figure}
	\centering
	\begin{tikzpicture}[scale=0.45]
	\draw[thick, ->] (-1, 0) -- (8, 0);
	\draw[thick, ->] (0, -3.5) -- (0, 5);
    \draw[thick] (0.2, -3) .. controls (2, 6) and (4, 0) ..(7, 0.2);
	\draw (8,0) node[below]{$t$};
	\draw (0, 0) node[below left]{$0$};
	\draw (0,5) node[left]{$\psi_{(w, z)(t)}$};
   \draw (3,-1) node[below]{\textrm{Case 1}};
    \end{tikzpicture}
    \hspace{.5cm}
    \begin{tikzpicture}[scale=0.45]
    \draw[thick, ->] (-1, 0) -- (8, 0);
    \draw[thick, ->] (0, -3.5) -- (0, 5);
	\draw[thick] (0.2 ,4.5) .. controls (0.6, 0.6) .. (7, 0.2) ;
	\draw (8,0) node[below]{$t$};
	\draw (0, 0) node[below left]{$0$};
	\draw (0,5) node[left]{$\psi_{(w, z)(t)}$};
\draw (3,-1) node[below]{\textrm{Case 2}};
	\end{tikzpicture}	
    \begin{tikzpicture}[scale=0.45]
	\draw[thick, ->] (-1, 0) -- (8, 0);
	\draw[thick, ->] (0, -3) -- (0, 4);
    \draw[thick] (0 ,0) .. controls (0.3,0.3) and (0.5,-1.5) .. (1, -1.5) .. controls (2, -1.5) and (2, 1.5) .. (3, 1.5) .. controls (3.7, 1.5) and (4, 0)  .. (7, -2); 
	\draw (8,0) node[below]{$t$};
	\draw (0, 0) node[below left]{$0$};
	\draw (0,4) node[left]{$\phi_{(w, z)(t)}$};
\draw (4,-2) node[below]{\textrm{Case 1}};
	\end{tikzpicture}
	\hspace{.5cm}
	\begin{tikzpicture}[scale=0.45]
   \draw[thick, ->] (-1, 0) -- (8, 0);
	\draw[thick, ->] (0, -3) -- (0, 4);
	\draw[thick] (0 ,0.0) .. controls (0.9, -0.2) and (2,5) .. (7, -2);
	\draw (8,0) node[below]{$t$};
	\draw (0, 0) node[below left]{$0$};
	\draw (0,4) node[left]{$\phi_{(w, z)(t)}$};
   \draw (4,-2) node[below]{\textrm{Case 2}};
    \end{tikzpicture}
\end{figure}
Thus from above discussion we have the following Lemma.
\begin{lem}\label{L37} 
\noindent (i) If $\lambda \ds \int_{\Omega} f(x)|w(x, 0)|^qdx+\mu \ds \int_{\Omega} g(x)|z(x, 0)|^qdx >0 $,  there exist unique $t^+((w,z))<t_{*}<t^-((w,z))$ such that $(t^\pm w,t^\pm z)\in \mathcal{N}_{\lambda,\mu}^\pm$ and $\mathcal{I}_{\lambda,\mu}(t^+w,t^+z)=\displaystyle\min_{0\leq t\leq t^-} \mathcal{I}_{\lambda,\mu}(tw,tz)$, $\mathcal{I}_{\lambda,\mu}(t^-w,t^-z) = \displaystyle \max_{t\geq t_{*}} \mathcal{I}_{\lambda,\mu}(tw,tz)$.\\
(ii) If $\lambda \ds \int_{\Omega} f(x)|w(x, 0)|^qdx+\mu \ds \int_{\Omega} g(x)|z(x, 0)|^qdx <0$, there exists a unique $\hat t((w, z))>0$ such that $(\hat tw,\hat tz)\in \mathcal{N}_{\lambda,\mu}^-$ and $\mathcal{I}_{\lambda,\mu}(\hat tw,\hat tz)=\displaystyle\max_{ t\geq 0} \mathcal{I}_{\lambda,\mu}(tw_,tz)$.
\end{lem}
\noindent The behavior of $\psi_{(w,z)}(t)$ and $\phi_{(w,z)}(t)$ in different cases has been shown in the above figures.  \\

\noi Concerning the component set $\mc N_{\lambda, \mu}^-$, we have the following {lemma} which helps us to show that the set $\mc N_{\lambda, \mu}^-$ is closed in the $\mc H(\mc C)-$ topology.
\begin{lem}\label{nlw}
There exists $\delta>0$ such that $\|(w, z)\|\geq \delta$ for all $(w, z)\in \mc N_{\lambda, \mu}^-$.
\end{lem}
\begin{proof}
Let $(w, z)\in \mc N_{\lambda, \mu}^-$ then from \eqref{phid2}, we get
\begin{align*}
&a(\|w\|^2+\|z\|^2)+3b(\|w\|^4+\|z\|^4)-{\lambda}{(q-1)}\int_{\Omega }f(x)|w(x,0)|^qdx\\&-{\mu(q-1)}\int_{\Omega }g(x)|z(x,0)|^qdx<2(2^*_s-1)\int_{\Omega }|w(x,0)|^{\alpha}|z(x, 0)|^\beta dx.
\end{align*}
 Now using \eqref{22nlam} together with \eqref{consys}, we get
 \begin{align*}
a(2-q)(\|w\|^2+\|z\|^2)&<2(2^*_s-q)\int_{\Omega }|w(x,0)|^{\alpha}|z(x, 0)|^\beta dx,\\
 &<2(2^*_s-q)(\kappa_sS(s, \alpha, \beta))^\frac{-2^*_s}{2}(\|w\|^2+\|z\|^2)^\frac{2^*_s}{2}
 \end{align*}
 which implies that $\|(w, z)\|^{2^*_s-2}>C$. Hence $\|(w, z)\|\geq \delta$ for some $\delta>0$.
 \end{proof}
\begin{cor}\label{nlclosed}
{Assume $(\lambda,\mu)\in \mc M_{_{\Gamma_1}}$. Then} $\mc N_{\lambda, \mu}^-$ is closed set in the $\mc H(\mc C)-$ topology.
\end{cor}
\begin{proof}
Let $\{(w_k,z_k)\}$ be a sequence in $\mc N_{\lambda, \mu}^-$ such that $(w_k, z_k)\rightarrow (w, z)$ in $\mc H(\mc C)$. {Since  $\mc N_{\lambda, \mu}^0=\emptyset$} for $(\lambda,\mu)\in \mc M_{_{\Gamma_1}}$ , $(w, z)\in \overline {\mc N_{\lambda, \mu}^-}=\mc N_{\lambda, \mu}^-\cup\{0\}$. Now using Lemma \ref{nlw}, we get $\|(w, z)\|=\displaystyle \lim_{k\rightarrow \infty}\|(w_k, z_k)\|\geq \delta>0$. Hence $(w, z)\neq (0, 0)$. Therefore $(w, z)\in \mc N_{\lambda, \mu}^-$.
\end{proof}

\noi Now we state the following Lemma {providing a local parametrization around any point of $\mathcal{N}_{\lambda,\mu}$}.

\begin{lem} \label{zii}
For a given $(w,z) \in \mathcal{N}_{\lambda,\mu}$ and $(\lambda,\mu) \in \mc M_{_{\Gamma_1}},$ there exists {$\delta > 0$} and a differentiable function
{$\xi : \mathcal{B}(0,\delta) \subseteq \mathcal H(\mc C) \rightarrow \mathbb{R}$} such that $\xi(0)=1,$ the function $\xi(v_1,v_2)(w-v_1,z-v_2)\in \mathcal{N}_{\lambda,\mu}$
and $\langle\xi^{'}(0), (v_1,v_2)\rangle=\frac{N}{D}$ for all $(v_1, v_2)\in \mc H(\mc C)$, where
\begin{align*}
N&=2a (\langle w, v_1\rangle+\langle z, v_2\rangle)+ 4b (\|w\|^{2}\langle w, v_1\rangle+\|z\|^{2}\langle z, v_2\rangle)\\&\quad -q\lambda \int_{\Omega} f(x)|w(x, 0)|^{q-2}w(x, 0)v_1(x, 0)dx\\&\quad - q\mu \int_{\Omega} g(x)|z(x, 0)|^{q-2}z(x, 0)v_2(x, 0)dx)\\&\quad- 2\alpha\int_{\Omega} |w(x, 0)|^{\alpha-2}w(x, 0)|z(x, 0)|^\beta v_1(x, 0) dx\\&\quad -2\beta \int_{\Omega } |w(x, 0)|^{\alpha}|z(x, 0)|^{\beta-2}z(x, 0) v_2(x,0) dx\\
 D&=(2^*_s-2)a(\|w\|^{2}+\|z\|^{2}) + (2^*_s-4)b(\|w\|^{4}+\|z\|^{4})\\
 &\quad -(2^*_s-q)\left(\lambda\int_{\Omega} f(x)|w(x,0)|^qdx+\mu\int_{\Omega} g(x)|z(x, 0)|^qdx\right).
\end{align*}

\end{lem}

\begin{proof}
{The proof is quite standard. We argue as in Lemma 3.7 in \cite{PDH} or Lemma 2.6 in \cite{MR3474407}. Let
\begin{eqnarray*}
{\mathcal F}_{(w,z)}\,:\, \mathbb R^+\times \mathcal H(\mc C)\to\mathbb R
\end{eqnarray*}
defined by
\begin{eqnarray*}
&{\mathcal F}_{(w,z)}(t,(v_1,v_2))=ta\|(v_1,v_2)-(w,z)\|^2+t^3b\|(v_1,v_2)-(w,z)\|^4\\
&-t^{q-1}\lambda \ds\int_\Omega f(x)|v_1(x,0)-w(x,0)|^qdx-t^{q-1}\mu\ds \int_\Omega g(x)|v_2(x, 0)-z(x, 0)|^qdx\\
&-2t^{2^*_s-1}\ds \int_{\Omega}|v_1(x, 0)-w(x, 0)|^\alpha|v_2(x, 0)-z(x, 0)|^\beta dx,
\end{eqnarray*}
for any $\left(t,(v_1,v_2)\right)\in  \mathbb R^+\times\mathcal H(\mc C)$. Observing that
\begin{equation*}
{\mathcal F}_{(w,z)}(1,(0,0))=0\mbox{ and since ${\mathcal N}^0_{\lambda,\mu}=\emptyset$, } \displaystyle\frac{\partial F}{\partial t}(1,(0,0))\neq 0,
\end{equation*}
one can apply the implicit function theorem to get the result.}
\end{proof}

\noi Now using the Lemma \ref{zii}, we prove the following proposition which can help us to extract a Palais-Smale sequence out of Nehari decompositions. 

\begin{prop}\label{prp1}
(i) Let $(\lambda, \mu) \in \mc M_{_{\Gamma_1}}$. Then there exists a minimizing sequence  $\{(w_k, z_k)\} \subset \mathcal{N}_{\lambda, \mu}$ such that
\[
 \mathcal I_{\lambda, \mu}(w_{k}, z_k) = \Theta_{\lambda, \mu}+o_k(1)\;\textrm{ and}\; {\|\mathcal I_{\lambda, \mu}^{'}(w_{k}, z_k)\|_*= o_k(1).}
\]

\noindent (ii) Let $(\lambda, \mu) \in \mc M_{_{\Gamma_3}}$ for some $\Gamma_3>0$. Then there exists a minimizing sequence  $\{(w_k, z_k)\} \subset \mathcal{N}_{\lambda, \mu}^-$  such that
\[
   \mathcal I_{\lambda, \mu}(w_{k}, z_k) = \Theta_{\lambda, \mu}^-+o_k(1)\;\textrm{ and}\; {\|\mathcal I_{\lambda, \mu}^{'}(w_{k}, z_k)\|_*= o_k(1),}
\]
where $\|\;\|_*$ is a  norm of the dual space of $\mc H(\mc C)$
\end{prop}
\begin{proof} To avoid any repetition, we only prove the case $(i)$ of the above Proposition. The proof for the case $(ii)$ is similar. From Lemma \ref{le44}, $\mc I_{\lambda, \mu}$ is bounded below on $\mathcal N_{\lambda, \mu}$. So by Ekeland variational principle, there exists a minimizing sequence $\{(w_{k}, z_k)\}\in \mathcal N_{\lambda, \mu}$ such that
\begin{align}
\mathcal I_{\lambda, \mu}(w_{k}, z_k)&\leq \Theta_{\lambda, \mu}+\frac{1}{k},\label{Pek1}\\
\mathcal I_{\lambda, \mu}(\bar w, \bar z)&\geq \mc I_{\lambda, \mu}(w_{k}, z_k)- \frac{1}{k}\|(\bar w-w_k,\bar z-z_k)\|\;\;\mbox{for all}\;\;(\bar w, \bar z)\in \mc N_{\lambda, \mu}.\nonumber
\end{align}
Using equation \eqref{Pek1}, it is easy to show that $(w_{k}, z_k)\not\equiv 0$. Indeed, using \eqref{Pek1} and  Lemma \ref{3a}, we get 
\begin{equation}\label{psre12}
\begin{aligned}
&\frac{\Theta_{\lambda, \mu}}{2}\geq \mathcal I_{\lambda, \mu}(w_{k}, z_k)\geq \left(\frac{1}{2}-\frac{1}{2^*_s}\right)a(\|w_k\|^{2}+\|z_k\|^{2})\\
&-\left(\frac{1}{q}-\frac{1}{2^*_s}\right)\left(\lambda\int_{\Omega }  f(x)|w_k(x, 0)|^{q}dx+\mu \int_{\Omega }  g(x)|z_k(x, 0)|^{q}dx\right)
\end{aligned}
\end{equation}
which implies
\begin{equation}\label{psre13}
\begin{aligned}
&-\frac{2^*_sq}{2}\Theta_{\lambda, \mu}\leq \lambda\int_{\Omega }  f(x)|w_k(x, 0)|^{q}dx+\mu \int_{\Omega }  g(x)|z_k(x, 0)|^{q}dx\\
&\quad\quad\leq 
 (\kappa_s S(s, N))^\frac{-q}{2}\left((\lambda\|f\|_\gamma)^\frac{2}{2-q}+(\mu\|g\|_\gamma)^\frac{2}{2-q}\right)^\frac{2-q}{2}
  \|(w_k, z_k)\|^{q}.
\end{aligned}
\end{equation}
From \eqref{psre13}, we get immediately  
\[
\|(w_k, z_k)\|\geq \left( -\frac{2^*_sq}{2}\Theta_{\lambda, \mu}(\kappa_s S(s, N))^\frac{q}{2}\left((\lambda\|f\|_\gamma)^\frac{2}{2-q}+(\mu\|g\|_\gamma)^\frac{2}{2-q}\right)^\frac{q-2}{2}\right)^\frac{q}{2}.
\]
Besides this, combining \eqref{psre12} and \eqref{psre13}, we obtain 
\[
\|(w_k, z_k)\|\leq \left( \frac{2(2^*_s-q)}{q(2^*_s-2)}(\kappa_s S(s, N))^\frac{-q}{2}\left((\lambda\|f\|_\gamma)^\frac{2}{2-q}+(\mu\|g\|_\gamma)^\frac{2}{2-q}\right)^\frac{2-q}{2}\right)^\frac{1}{2-q}.
\]
 Next we claim that $\|\mathcal I_{\lambda, \mu}^\prime(w_{k}, z_k)\|_*\rightarrow 0$ as $k \rightarrow 0$.
Now, using the Lemma \ref{zii} we get the differentiable functions {$\xi_k:\mathcal{B}(0, \delta_k)\rightarrow \mathbb{R}$} for some {$\delta_k>0$} such that $\xi_k(\bar w, \bar z)(w_k-\bar w, z_k-\bar z)\in \mc {N}_{\lambda, \mu}$,\; $\textrm{for all}\;\; (\bar w, \bar z)\in \mathcal{B}(0,{ \delta_k}).$
For fixed $k$, choose {$0<\rho<\delta_k$}. Let $(w, z)\in \mc H(\mc C)$ with $(w, z)\not\equiv 0$ and let $(\bar w, \bar z)_\rho=\frac{\rho (w, z)}{\|(w, z)\|}$. We set $\eta_\rho=\xi_k((\bar w, \bar z)_\rho)((w_k, z_k)-(\bar w, \bar z)_\rho)$. Since $\eta_\rho \in \mc {N}_{\lambda, \mu}$, we get from  \eqref{22nlam}
\begin{align*}
\mathcal I_{\lambda, \mu}(\eta_\rho)-\mathcal I_{\lambda, \mu}(w_{k}, z_k)\geq-\frac{1}{k}\|\eta_\rho-(w_k,z_k)\|.
\end{align*}
Now by mean value theorem, we get
\begin{equation*}
\left\langle \mathcal I_{\lambda, \mu}^{\prime}(w_{k}, z_k),(\eta_\rho-(w_k,z_k))\right\rangle_*+o_k(\|(\eta_\rho-(w_k,z_k))\|)\geq -\frac{1}{k}\|(\eta_\rho-(w_k,z_k))\|.
\end{equation*}
Hence
\begin{align*}
&\left\langle \mathcal I_{\lambda, \mu}^{\prime}(w_{k}, z_k),-(\bar w ,\bar z)_\rho\right\rangle_*+(\xi_k(\bar w, \bar z)_\rho-1)\left\langle \mathcal I_{\lambda, \mu}^{\prime}(w_{k}, z_k),(w_{k}, z_k)-(\bar w, \bar z)_\rho\right\rangle_*\\
 &\geq -\frac{1}{k}\|(\eta_\rho -(w_k, z_k))\|+o(\|(\eta_\rho -(w_k, z_k))\|)
\end{align*}
and since $\left \langle \mathcal I_{\lambda, \mu}^{\prime}(\eta_\rho),((w_{k}, z_k)-(\bar w, \bar z)_\rho)\right\rangle_*=0$ by the fact that $\eta_\rho \in \mc {N}_{\lambda, \mu}$, we have
\begin{align*}
-\rho\left \langle \mathcal I_{\lambda, \mu}^{\prime}(w_{k}, z_k),\frac{(w, z)}{\|(w, z)\|}\right \rangle_*&+(\xi_k((\bar w, \bar z)_\rho)-1)\left\langle \mathcal I_{\lambda, \mu}^{\prime}(w_{k}, z_k)\right.\\&-\left.\mathcal I_{\lambda, \mu}^{\prime}(\eta_\rho),((w_{k}, z_k)-(\bar w, \bar z)_\rho)\right\rangle_*\\&\geq -\frac{1}{k}\|(\eta_\rho-(w_{k}, z_k))\|+o(\|(\eta_\rho-(w_{k}, z_k))\|).
\end{align*}
Thus
\begin{equation}\label{fifte}
\begin{aligned}
&\left\langle \mathcal I_{\lambda, \mu}^{\prime}(w_{k}, z_k),\frac{(w, z)}{\|(w, z)\|}\right\rangle_* \leq \frac{1}{k\rho}\|(\eta_\rho-(w_k, z_k))\|+\frac{o(\|(\eta_\rho-(w_k, z_k))\|)}{\rho}\\&+
\frac{(\xi_k((\bar w, \bar z)_\rho)-1)}{\rho}\left\langle  \mathcal I_{\lambda, \mu}^{\prime}(w_{k}, z_k)-\mathcal I_{\lambda, \mu}^{\prime}(\eta_\rho),((w_{k}, z_k)-(\bar w, \bar z)_\rho)\right\rangle_*.
\end{aligned}
\end{equation}
Since
$\displaystyle
\|(\eta_\rho-(w_{k}, z_k))\|\leq \rho|\xi_k((\bar w, \bar z)_\rho)|+|\xi_k((\bar w, \bar z)_\rho)-1|\|(w_{k}, z_k)\|
$
and
$
\displaystyle\lim_{\rho\rightarrow 0^+}\frac{|\xi_k((\bar w, \bar z)_\rho)-1|}{\rho}\leq \|\xi_k'(0)\|_*,$  taking limit  $\rho\rightarrow 0^+$\ in \eqref{fifte} and using \eqref{psre12} and \eqref{psre13}, we get
\begin{equation*}
\left\langle \mathcal I_{\lambda, \mu}^\prime(w_{k}, z_k),\frac{(w, z)}{\|(w, z)\|}\right\rangle_*\leq\frac{C}{k}(1+\|\xi_k^{'}(0)\|_*)
\end{equation*}
for some constant $C>0$, independent of $(w, z)$. So if we can show that $\|\xi_k^{'}(0)\|_*$ is bounded then we are done.
 Now from Lemma \ref{zii}, 
  using the boundedness of $\{(w_{k}, z_k)\}$ and H\"older's inequality, for some $K>0$, we get
\begin{align*}
\left\langle \xi^{\prime}(0),(\bar w, \bar z)\right\rangle_*= \frac{K\|(\bar w, \bar z)\|}{D},
\end{align*}
where $D$ is defined in Lemma \ref{zii}. Therefore, to prove the claim, we only need to prove that the denominator in the above expression is bounded away from zero. Suppose not. Then there exists a subsequence, still denoted by $\{(w_{k}, z_k)\}$, such that
\begin{align}\nonumber
&(2^*_s-2)a(\|w_k\|^{2}+\|z_k\|^{2}) + (2^*_s-4)b(\|w_k\|^{4}+\|z_k\|^{4})\\\label{baw}
 &-(2^*_s-q)\left(\lambda\int_{\Omega} f(x)|w_k(x,0)|^qdx+\mu\int_{\Omega} g(x)|z_k(x, 0)|^qdx\right)=o_k(1).
 \end{align}
From equation \eqref{baw} we get $\mc T_{\lambda, \mu}(w_{k}, z_k)=o_k(1)$. Now {following the proof of Lemma \ref{l22.2} we get $\mc T_{\lambda, \mu} (w_{k}, z_k)\geq C_1>0$ for all $k$ for some $C_1>0$ and $(\lambda, \mu)\in \mc M_{_{\Gamma_1}}$, which is a contradiction.}
\end{proof}
\section{Compactness of Palais-Smale sequences}

\noi Now we prove the following proposition which shows the compactness of Palais-Smale sequence {at sub-critical energy levels:}
\begin{prop}\label{crcmp}
Suppose $\{(w_k,z_k) \}$ be a sequence in $\mc H(\mathcal C)$ such that \\
\[
\mathcal I_{\lambda, \mu}(w_k, z_k)\rightarrow c \;\;\textrm{and}\;\; \mathcal I_{\lambda, \mu}^{\prime}(w_k, z_k)\rightarrow 0,\] where
\begin{equation}\label{clambda}
\begin{aligned}
&c<c_{\lambda, \mu}=\frac{2s}{N}\left(\frac12 a\kappa_sS(s,\alpha, \beta)\right)^\frac{N}{2s}-\\
&\left((\lambda\|f\|_\gamma)^\frac{2}{2-q}+(\mu\|g\|_\gamma)^\frac{2}{2-q}\right)^\frac{2(2-q)}{4-q}\frac{(4-q)(2^*_s-q)^\frac{4}{4-q}}{2(2^*_sq)}\left(\frac{b(\kappa_sS(s, N))^{-2}}{(2^*_s-4)}\right)^\frac{q}{4-q}
\end{aligned}
\end{equation}
is a positive constant, then there exists a strongly convergent subsequence.
\end{prop}
\begin{proof}
Let $\{(w_k, z_k)\}$ be a non-negative ( since $\mathcal I_{\lambda, \mu}(|w|,|z|)=\mathcal I_{\lambda, \mu}(w,z)$) $(PS)_c$ sequence for $\mc I_{\lambda, \mu}$ in $\mc H(\mathcal C)$ then it is easy to
 see that $\{(w_k, z_k)\}$ is bounded in $\mc H(\mathcal C)$. Therefore there exists non-negative 
 $(w_0, z_0)\in \mc H(\mathcal C)$ such that  
\begin{equation}\label{e2.4}
\begin{aligned}
w_k\rightharpoonup w_0 &\text{ and } z_k\rightharpoonup z_0\mbox{ in } L^{2^*_s}(\Omega \times\{0\}),\\
 \|w_k\|\rightarrow \sigma_1 &\text{ and } \|z_k\|\rightarrow \sigma_2, \text{ with } {\sigma^2=\sigma_1^2+\sigma_2^2} \\
w_k\rightarrow w_0 &\text{ and } z_k\rightarrow z_0\mbox{ in } L^p(\Omega\times\{0\})\text{ for any }p\in[1,2^*_s)\\
w_k(x, 0)\to w_0(x, 0)& \text{ and } z_k(x, 0)\to z_0(x, 0)\text{ a.e.  in }\Omega\times \{0\},\\
w_k\leq h_1 &\text{ and } z_k\leq h_2\text{ a.e.  in }\Omega \times \{0\},
\end{aligned}
\end{equation}
as $k\to\infty$, with $h_1, h_2\in L^p(\Omega\times\{0\})$ for a fixed $p\in[1,2^*_s)$.
If $\sigma=0$, we immediately see that $(w_k, z_k)\to (0, 0)$ in $\mathcal H(\mathcal C)$ as $k\to\infty$. Hence, let us assume that $\sigma>0$.

By \eqref{e2.4} and Brezis-Lieb Lemma \cite[Theorem 2]{MR699419} {(see also Lemma 8 in \cite{MR1698537})}, we have
\begin{align}\label{e2.9}
\|w_k\|^2&=\|w_k-w_0\|^2+\|w_0\|^2+o_k(1),\\\label{e2.9Inter}
\|z_k\|^2&=\|z_k-z_0\|^2+\|z_0\|^2+o_k(1),\\
\nonumber
\int_{\Omega}|w_k(x, 0)|^{\alpha}|z_k(x, 0)|^\beta dx&=\int_{\Omega}|w_k(x, 0)-w_0(x, 0)|^{\alpha}|z_k(x, 0)-z_0(x, 0)|^\beta dx
\\&+\int_{\Omega}|w_0|(x, 0)^{\alpha}|z_0(x, 0)|^\beta dx+o_k(1)\label{e2.09}
\end{align}
as $k\to\infty$.
Consequently, by $\mathcal{I}_{\lambda, \mu}^\prime(w_k, z_k)(w_k-w_0, z_k-z_0)=0$  as $k\to\infty$  together with
 \eqref{e2.4}, \eqref{e2.9}, \eqref{e2.09} and $\textbf{(fg)}$, we get
\begin{align*}
&(a+b\sigma_1^2)(\sigma_1^2-\|w_0\|^2)+(a+b\sigma_2^2)(\sigma_2^2-\|z_0\|^2)-\\
&\quad-\frac{2\alpha}{\alpha+\beta}\int_\Omega \left|w_k(x, 0)\right|^{\alpha-2}w_k(x, 0)|z_k(x, 0)|^\beta(w_k(x, 0)-w_0(x, 0)) dx\\&\quad-\frac{2\beta}{\alpha+\beta}\int_\Omega \left|w_k(x, 0)\right|^{\alpha}|z_k(x, 0)|^{\beta-2}z_k(x, 0)(z_k(x, 0)-z_0(x, 0))dx=o_k(1)\\
\end{align*}
which leads us us to 
\begin{equation}\label{I}
\begin{aligned}
(a+b\sigma_1^{2})&\lim_{k\to\infty}\|w_k-w_0\|^2+(a+b\sigma_2^{2})\lim_{k\to\infty}\|z_k-z_0\|^2\\=&2\lim_{k\to\infty}\int_{\Omega}|w_k(x, 0)-w_0(x, 0)|^{\alpha}|z_k(x, 0)-z_0(x, 0)|^\beta dx.
\end{aligned}
\end{equation}
{
Now, let us denote
{\begin{equation}\label{crlg}
\lim_{k\to\infty}\int_{\Omega}|w_k(x, 0)-w_0(x, 0)|^{\alpha}|z_k(x, 0)-z_0(x, 0)|^\beta dx=\ell^{2^*_s}.
\end{equation}}
Thus, from \eqref{I} and \eqref{crlg}, we get the following crucial formula
\begin{equation}\label{IF}
(a+b\sigma_1^{2})\lim_{k\to\infty}\|w_k-w_0\|^2+(a+b\sigma_2^{2})\lim_{k\to\infty}\|z_k-z_0\|^2=2\ell^{2^*_s}.
\end{equation}
Then, from \eqref{IF}, it is clear that $\ell\geq 0$. If $\ell=0$, since $\sigma>0$, by  \eqref{IF}, we have $(w_k, z_k)\to (w_0, z_0)$ in $\mathcal H(\mathcal C)$ as $k\to\infty$, concluding the proof.
Thus, let us assume by contradiction that $\ell>0$.}
By \eqref{consys} and \eqref{crlg},
{\begin{equation}\label{ukmuo}
\|w_k-w_0\|^2+\|z_k-z_0\|^2\geq \kappa_s S(s, \alpha, \beta)\ell^2.
\end{equation}}
Using the notation in \eqref{e2.4} and \eqref{IF}, together with \eqref{ukmuo}, we get
{\begin{equation}\label{ll}
2\ell^{2^*_s-2}\ge \kappa_s S(s, \alpha, \beta)(a+b\min(\sigma_1,\sigma_2)^{2}).
\end{equation}}
Noting that \eqref{IF} implies in particular that
\begin{equation}\label{ukuoai}
(a+b\sigma_1^2)\big(\sigma_1^2-\|w_0\|^2\big)+(a+b\sigma_2^2)\big(\sigma_2^2-\|z_0\|^2\big)= 2\ell^{2^*_s}.
\end{equation}
Using \eqref{ll} in \eqref{ukuoai} together with \eqref{ukmuo} it follows that
{
\begin{align*}
\big(2\ell^{2^*_s}\big)^{\frac{2^*_s-2}{2}}
&=\left((a+b\sigma_1^2)(\sigma_1^2-\|w_0\|^2)+(a+b\sigma_2^2)(\sigma_2^2-\|z_0\|^2)\right)^{\frac{2^*_s-2}{2}}\\
&\geq(a+b\min(\sigma_1,\sigma_2)^{2})^{\frac{2^*_s-2}{2}}\big(\ds \lim_{k\rightarrow \infty}\left(\|w_k-w_0\|^2+\|z_k-z_0\|^2\right)\big)^{\frac{2^*_s-2}{2}}\\
&\geq (a+b\min(\sigma_1,\sigma_2)^{2})^{\frac{2^*_s-2}{2}} (\kappa_sS(s, \alpha, \beta))^{\frac{2^*_s-2}{2}}\ell^{2^*_s-2}\\
&\geq \frac{1}{2}(a+b\min(\sigma_1,\sigma_2)^{2})^{\frac{2^*_s}{2}}(\kappa_sS(s, \alpha, \beta))^\frac{2^*_s}{2}.
\end{align*}}

From this, we obtain
\begin{equation}\label{bsub}
\begin{aligned}
(\sigma^2)^{\frac{2^*_s-2}{2}}&\ge\big((\sigma_1^2-\|w_0\|^2)+(\sigma_2^2-\|z_0\|^2)\big)^{\frac{2^*_s-2}{2}}\\&\ge
\frac{1}{2}(\kappa_sS(s, \alpha, \beta))^\frac{2^*_s}{2}(a+b\min(\sigma_1,\sigma_2)^{2})
\end{aligned}
\end{equation}
so that, we have
\begin{equation}\label{bsubr}
\begin{aligned}
\sigma^2&\geq (\kappa_sS(s, \alpha, \beta))^{\frac{N}{2s}}(\frac{1}{2}(a+b\min(\sigma_1,\sigma_2)^{2}))^\frac{2}{2^*_s-2}\\&\geq
(\kappa_sS(s, \alpha, \beta))^\frac{N}{2s}(\frac{a}{2})^\frac{2}{2^*_s-2}.
\end{aligned}
\end{equation} 
 Now, 
\begin{align*}
c&= \mc I_{\lambda, \mu}(w_k, z_k)-\frac{1}{2^*_s}\langle \mc I_{\lambda, \mu}^{\prime}(w_k, z_k), (w_k, z_k)\rangle+o_k(1)\\
&=\left(\frac12-\frac{1}{{2^*_s}}\right)a(\|w_k\|^{2}+\|z_k\|^{2})
+\left(\frac{1}{4}-\frac{1}{{2^*_s}}\right)b(\|w_k\|^{4}+\|z_k\|^{4})\\
&\quad-\left(\frac{1}{q}-\frac{1}{2^*_s}\right)\left(\lambda\int_{\Omega }  f(x)|w_k(x, 0)|^{q}dx+\mu \int_{\Omega }  g(x)|z_k(x, 0)|^{q}dx\right)+o_k(1)\\
&\geq \left(\frac12-\frac{1}{{2^*_s}}\right)a(\|w_k\|^{2}+\|z_k\|^{2})
+\left(\frac{1}{4}-\frac{1}{{2^*_s}}\right)b(\|w_k\|^{4}+\|z_k\|^{4})\\
&
 \quad-\left(\frac{1}{q}-\frac{1}{{2^*_s}}\right)(\kappa_s S(s, N))^\frac{-q}{2}\left((\lambda\|f\|_\gamma)^\frac{2}{2-q}+(\mu\|g\|_\gamma)^\frac{2}{2-q}\right)^\frac{2-q}{2}\times\\&\quad\quad\quad
  (\|w_k\|^{q}+ \|z_k\|^{q})+o_k(1).
\end{align*}
 Define
 \begin{align*}
F_b(t)&=\left(\frac{1}{4}-\frac{1}{{2^*_s}}\right)bt^{4}\\&
 - \left(\frac{1}{q}-\frac{1}{{2^*_s}}\right)(\kappa_s S(s, N))^\frac{-q}{2}\left((\lambda\|f\|_\gamma)^\frac{2}{2-q}+(\mu\|g\|_\gamma)^\frac{2}{2-q}\right)^\frac{2-q}{2}t^{q}
  \end{align*}
 and using the critical point arguments, we obtain
\begin{align*}
 F_b(t)&\geq -\left((\lambda\|f\|_\gamma)^\frac{2}{2-q}+(\mu\|g\|_\gamma)^\frac{2}{2-q}\right)^\frac{2(2-q)}{4-q}\times\\&\frac{(4-q)(2^*_s-q)^\frac{4}{4-q}}{4(2^*_sq)}\left(\frac{b(\kappa_sS(s, N))^{-2}}{(2^*_s-4)}\right)^\frac{q}{4-q}.
\end{align*}
 Hence using this estimate, we get
 \begin{align*}
c&+o_k(1)\geq \left(\frac12-\frac{1}{{2^*_s}}\right)a(\|w_k\|^{2}+\|z_k\|^{2})-\\&\left((\lambda\|f\|_\gamma)^\frac{2}{2-q}+(\mu\|g\|_\gamma)^\frac{2}{2-q}\right)^\frac{2(2-q)}{4-q}\frac{(4-q)(2^*_s-q)^\frac{4}{4-q}}{2(2^*_sq)}\left(\frac{b(\kappa_sS(s, N))^{-2}}{(2^*_s-4)}\right)^\frac{q}{4-q}.
 \end{align*}
Now, as $k\to\infty$, by  \eqref{bsubr}, we have
\begin{align*}
c&\geq\frac{2s}{N}\left(\frac12a\kappa_sS(s,\alpha, \beta)\right)^\frac{N}{2s}-\\
&\left((\lambda\|f\|_\gamma)^\frac{2}{2-q}+(\mu\|g\|_\gamma)^\frac{2}{2-q}\right)^\frac{2(2-q)}{4-q}\frac{(4-q)(2^*_s-q)^\frac{4}{4-q}}{2(2^*_sq)}\left(\frac{b(\kappa_sS(s, N))^{-2}}{(2^*_s-4)}\right)^\frac{q}{4-q}
 \end{align*}
which contradicts the assumption $c<c_{\lambda, \mu}$, considering \eqref{clambda}. This concludes the proof.\end{proof}
\section{Existence of first solution}
In this section we prove Theorem \ref{22mht1} $(i)$, by a minimization argument on $\mc N_{\lambda, \mu}$. Note that the $\Theta_{\lambda, \mu}<0$ from Lemma \ref{3a}.\\
\noindent \textbf{Proof of Theorem \ref{22mht1} (i)} Let us fix $\Gamma_4>0$ such that $c_{\lambda, \mu}>0$ for $(\lambda, \mu)\in \mathcal M_{_{\Gamma_4}}$.
 Assume $\Gamma_0=\min\{\Gamma_1, \Gamma_2, \Gamma_4\}$. Now as the functional $\mc I_{\lambda, \mu}$ is bounded below in $\mathcal N_{\lambda, \mu}$, we minimize  $\mc I_{\lambda, \mu}$
 in $\mathcal N_{\lambda, \mu}$. Using Proposition \ref{prp1} $(i)$, we get a minimizing Palais-Smale sequence such that $\mc I_{\lambda, \mu}(w_k, z_k)\rightarrow \Theta_{\lambda, \mu}$. Then it is easy to show that $\{(w_k, z_k)\}$ is bounded and therefore there exists $(w_0, z_0)\in \mc H(\mc C)$ such that $(w_k, z_k)\rightharpoonup (w_0, z_0)$ in $\mc H(\mc C)$. Now  for $(\lambda, \mu) \in \mathcal M_{_{\Gamma}}$, where $\Gamma\in(0, \Gamma_0)$ and using Lemma \ref{3a} and Proposition \ref{crcmp} we get  that $(w_k, z_k)\rightarrow (w_0, z_0)$ in $\mc H(\mc C)$ which implies that the $(w_0, z_0)$ is a minimizer of  $\mc I_{\lambda, \mu}$ in $\mc N_{\lambda, \mu}$  with $\mc I_{ \lambda, \mu}(w_0, z_0)<0$.
  Now we discuss some properties of this minimizer $(w_0, z_0).$\\
  
\noindent (i) $(w_0, z_0) \in \mathcal N_{\lambda, \mu}^+$ for $\Gamma\in (0,\Gamma_0)$.
\begin{proof}
 If not then $(w_0, z_0)\in \mc N_{\lambda, \mu}^-$. Note that using $(w_0, z_0)\in \mc N_{\lambda, \mu}$ and $\mc I_{ \lambda, \mu}(w_0, z_0)<0$ we get
 \[
 \lambda\ds \int_{\Omega} f(x)|w_0(x, 0)|^qdx+\mu\ds \int_{\Omega} g(x)|z_0(x, 0)|^qdx>0.
 \]
  Therefore from Lemma \ref{L37}, we get unique $t^-((w_0, z_0))>t^+((w_0, z_0))>0$ such that $t^-(w_0, z_0)\in \mc N_{\lambda, \mu}^-$ and $t^+(w_0, z_0)\in \mc N_{\lambda, \mu}^+$ which implies $t^-=1$ and $t^+<1$. Therefore we can find $t_0\in (t^+, t^-)$ such that
\begin{align*}
 \mathcal I_{\lambda, \mu}(t^+(w_0, z_0))&=\displaystyle\min_{0\leq t\leq t^-} \mathcal I_{\lambda, \mu}(t(w_0, z_0))<\mathcal I_{\lambda, \mu}(t_0(w_0, z_0))\\&\leq \mathcal I_{\lambda, \mu}(t^-(w_0, z_0))=\mathcal I_{\lambda, \mu}(w_0, z_0)=\Theta_{\lambda,\mu}
 \end{align*}
 which is a contradiction. Hence $(w_0, z_0)\in \mc N_{\lambda, \mu}^+.$ As $\mc I_{\lambda, \mu}(w, z)=\mc I_{ \lambda, \mu}(|w|, |z|)$, we can assume that $w_0\geq, z_0\ge 0$.
 \end{proof}
 
 \noindent (ii) $(w_0, z_0)$ is not semi trivial, that is,  $w_0\not\equiv0, z_0\not\equiv0.$
\begin{proof}
Suppose by contradiction that $z_0\equiv0.$ Then $w_0$ is a nontrivial non-negative solution of
\begin{equation*}
\left\{\begin{array}{rll}
 -\mathrm{div} (y^{1-2s}\nabla w)&=0, &\textrm{in}\;\; \mathcal{C},\\
 w &=0 &\text{on}\;\; \partial_L,\\
  M(\|w\|^2)\frac{\partial w}{\partial \nu}&= \lambda f(x)|w|^{q-2}w& \textrm{on}\;\; \Omega\times \{0\},
\end{array}
\right.
\end{equation*}
By the maximum principle \cite{MR2825595} we get $w_0>0$
in $E^1_{0}(\mathcal {C})$ and
{
\begin{eqnarray}\label{nonsemitrivial}
M(\|w_0\|^2)\|w_0\|^2=\lambda \int_{\Omega}f(x)|w_0(x, 0)|^qdx>0.
\end{eqnarray}
Moreover, we can
choose $\tilde z_0\in E^1_{0}(\mathcal {C})\backslash\{0\}$ such
that
{\begin{eqnarray*}
\int_{\Omega }g(x)|\tilde z_0(x, 0)|^qdx>0.
\end{eqnarray*}
Since $0<q<2$ and from the mean value theorem, there exists $t_0>0$ such that $z_0^*=t_0\tilde z_0$ satisfies
\begin{eqnarray}\label{non-semitrivial}
M(\|z_0^*\|^2)\|z_0^*\|^2=\mu \int_{\Omega }g(x)|z_0^*(x, 0)|^qdx>0.
\end{eqnarray}}
Therefore,
\begin{equation}\label{posin1}
0< a(\|w_0\|^2+ \|z_0^*\|^2)\leq \lambda \int_{\Omega}f(x)|w_0(x, 0)|^qdx+\mu \int_{\Omega }g(x)|z_0^*(x, 0)|^qdx.
\end{equation}
 Hence by Lemma \ref{L37} there exist unique
$0<t^+<{t_*}<t^-$ such that $(t^+w_0,t^+z_0^*)\in
\mathcal {N}_{\lambda,\mu}^+$
and
$$\mathcal{I}_{\lambda,\mu}(t^+w_0,t^+z_0^*)=\inf_{0\leq
t\leq t^-}\mc I_{\lambda,\mu}(tw_0,tz_0^*).$$
 From \eqref{siutd}, one can show that the positive zero $t_*$ of $\psi^\prime_{w, z}(t)=0$ satisfies the following
\[
t_{*} \geq \left(\frac{(2^*_s-q)\left(\lambda\ds \int_{\Omega} f(x)|w_0(x, 0)|^qdx+\mu\ds \int_{\Omega} g(x)|z_0^*(x, 0)|^qdx\right)}{a(2^*_s-2)(\|w_0\|^{2}+\|z_0^*\|^{2})}\right)^{\frac{1}{2-q}}.
\]
Using \eqref{posin1}, it is easy to see that $t_*>1$. {Using \eqref{nonsemitrivial} and \eqref{non-semitrivial}} we have
\begin{eqnarray*}
\mc I_{\lambda,\mu}(w_0,z_0^*)<
\mc I_{\lambda,\mu}(w_0,0).
\end{eqnarray*}}
Therefore,
$$
\Theta^+_{\lambda,\mu}\leq
\mc I_{\lambda,\mu}(t^+w_0,t^+z_0^*)\leq
\mc I_{\lambda,\mu}(w_0,z_0^*)<
\mc I_{\lambda,\mu}(w_0,0)=\Theta^+_{\lambda,\mu}$$
which is a contradiction. Hence {$w_0\not \equiv 0 $ and $z_0\not \equiv 0 $}. Now using the fact that $M(t)$ is positive for $t>0$ and strong maximum principle (see Lemma 2.6 \cite{MR2825595}),  we get $w_0>0, z_0>0$.
\end{proof} 
\noindent(iii) $(w_0, z_0)$ is indeed a local minimizer of $\mc I_{\lambda, \mu}$ in $\mc H(\mathcal C)$.
\begin{proof} 
Since $(w_0,z_0) \in \mathcal N_{\lambda, \mu}^{+}$, we have $t^+=1
<t_*$. Hence by continuity of $(w,z)\mapsto t_*$, {for
$\ell>0$ small enough, there exists $\delta=\delta(\ell)>0$ such that $1+\ell< t_*(w_0-w, z_0-z)$
for all $\|(w, z)\|<\delta$}. Also, from Lemma \ref{zii}, for $\delta>0$
small enough, we obtain a $C^1$ map $t: \mc {B}(0,\delta)\rightarrow \mathbb R^+$
such that $t(w_0-w, z_0-z)\in  \mathcal N_{\lambda, \mu}$, $t(0)=1$. Therefore, for
{$\ell>0$ and $\delta=\delta(\ell)>0$ small enough,} we have {$t^+(w_0-w, z_0-z)=
t(w, z)<1+\ell<t_*(w_0-w, z_0-z)$ for all $\|(w, z)\|<\delta$}. Since $t_*(w_0-w, z_0-z)>1$,
we obtain $\mathcal I_{\lambda,\mu}(w_0, z_0)\leq \mathcal I_{\lambda, \mu}(t^+(w_0-w, z_0-z))\leq \mathcal I_{\lambda, \mu}(w_0-w, z_0-z)$
for all $\|(w, z)\|<\delta$. This shows that $(w_0, z_0)$ is a local minimizer for
$\mathcal I_{\lambda, \mu}$ in $\mc H(\mathcal C)$.
\end{proof} 

\section{Existence of second solution in $\mathcal N_{\lambda, \mu}^-$}
\noi Now we show the existence of second solution in $\mathcal N_{\lambda, \mu}^-$. 
Consider $\mc P=\{x\in \Omega\;|\;f(x)>0\}\cap \{x\in \Omega\;|\;g(x)>0\}$ has a positive measure. Consider the test functions as
$\eta \in C_c^{\infty}(\mathcal C_{\mc P})$, where {$\mathcal C_{\mc P}=\mc P\times [0, \infty)$} such that $0\leq \eta(x,y)\leq 1$
in $\mathcal C_{\mc P}$ and
\[
(\mathrm {supp}\;f^+\times \{y>0\})\cap (\mathrm {supp}\; g^+\times \{y>0\})\cap\{(x,y)\in \mathcal C_{\mc P}: \eta=1\}\neq \emptyset.
\]
Moreover, for $\rho>0$ small, $\eta(x,y)=1$ on {$\mc B_{\rho}(0)^+$} and $\eta(x,y)=0$ on $\mc B^c_{2\rho}(0)$. We take $\rho$ small enough such that $\mc B_{2\rho}(0)\subset \mathcal C_{\mc P}$. Consider $( w_{\epsilon,\eta, \alpha}=\eta \sqrt{\alpha} w_\epsilon, w_{\epsilon,\eta, \beta}=\eta \sqrt{\beta} w_\epsilon)\in \mc H(\mathcal C)$, where $w_\epsilon=E_s(u_\epsilon)$  and $u_\epsilon$ is defined  in Lemma \ref{22traceemb}.
Then, we have the following lemma.
\begin{lem}\label{II}
Let $(w_{0}, z_0)$ be the local minimum for the functional $\mc I_{\lambda, \mu}$ in $\mc H(\mc C)$. Then for every $r>0$ and a.e. $\eta \in C_c^{\infty}(\mathcal C_{\mc P})$\; there exist $\epsilon_{0} = \epsilon_{0}(r, \eta) > 0$ and $\Gamma_*>0$ such that
\begin{equation*}
\mathcal I_{\lambda}(w_{0}+r\;w_{\epsilon,\eta, \alpha}, z_{0}+r\;w_{\epsilon,\eta, \beta} ) <c_{\lambda, \mu}
\end{equation*}
for $b \in (0,\epsilon_{0}),\; \epsilon \in (0,\epsilon_{0}) $ and $(\lambda, \mu)\in \mc M_{_{\Gamma_*}}.$
\end{lem}
\begin{proof}
Using the definition of $\mc I_{\lambda, \mu}$, we open the proof as follows.
\begin{align*}  
&\mathcal I_{\lambda, \mu}(w_{0} + r\;w_{\epsilon,\eta, \alpha}, z_{0} + r\;w_{\epsilon,\eta, \beta})=\frac{a}{2}(\| w_{0} + r\;w_{\epsilon,\eta, \alpha}\|^{2}+\| z_{0} + r\;w_{\epsilon,\eta, \beta}\|^{2}) \\
&+ \frac{b}{4}(\| w_{0} + r\;w_{\epsilon,\eta, \alpha}\|^{4}+\| w_{0} + r\;w_{\epsilon,\eta, \beta}\|^{4})- \frac{\lambda}{q}\int_{\Omega} f(x)|w_{0}+ r\;w_{\epsilon, \eta, \alpha}|^{q}dx\\
&-\frac{\mu}{q}\int_{\Omega} g(x)|z_{0}+ r\;w_{\epsilon, \eta, \beta}|^{q}dx- \frac{2}{2^*_s}\int_{\Omega} |w_{0} + r\;w_{\epsilon,\eta, \alpha}|^\alpha |z_{0} + r\;w_{\epsilon,\eta, \beta}|^\beta dx\\
 \end{align*}
Using the fact that $(w_{0}, z_0)$ is a positive solution of problem $(S_{\lambda, \mu})$, we get
 \begin{align*}
 &\mathcal I_{\lambda,\mu}(w_{0} + r\;w_{\epsilon,\eta, \alpha}, z_{0} + r\;w_{\epsilon,\eta, \beta})\leq  \mc I_{\lambda, \mu}(w_0, z_0) +\frac{a}{2}r^{2}(\|w_{\epsilon, \eta, \alpha}\|^2+\|w_{\epsilon, \eta, \beta}\|^2)\\&+ \frac{b}{4}r^{4}(\|w_{\epsilon, \eta, \alpha}\|^{4}+\|w_{\epsilon, \eta, \beta}\|^{4})+\frac{3}{2}{b}r^{2}(\| w_{0}\|^{2}\|w_{\epsilon, \eta, \alpha}\|^2+\| z_{0}\|^{2}\|w_{\epsilon, \eta, \beta}\|^2)
 \\&+ b\;r^{3} (\| w_{0}\|\|w_{\epsilon, \eta, \alpha}\|^{3}+\| z_{0}\|\|w_{\epsilon, \eta, \beta}\|^{3})\\&
  -\frac{\lambda}{q} \int_{\Omega} f(x)( |w_{0} + r\;w_{\epsilon,\eta, \alpha}|^{q}-|w_0|^q-qr|w_{0}|^{q-1} w_{\epsilon,\eta, \alpha} )(x,0)dx\\&-\frac{\mu}{q} \int_{\Omega} g(x)( |z_{0} + r\;w_{\epsilon,\eta, \beta}|^{q}-|z_0|^q-qr|z_{0}|^{q-1} w_{\epsilon,\eta, \beta} )(x,0)dx\\
 &- \frac{2}{{2^*_\alpha}}\int_{\Omega} \left(|w_0+rw_{\epsilon,\eta, \alpha}|^{\alpha}|z_0+rw_{\epsilon,\eta, \beta}|^{\beta}-|w_0|^{\alpha}|z_0|^{\beta}\right.\\
 &\quad \quad \quad \quad\left. -{\alpha}r|w_0|^{\alpha-1}|z_0|^\beta w_{\epsilon,\eta, \alpha}-{\beta}r|w_0|^{\alpha}|z_0|^{\beta-1} w_{\epsilon,\eta, \beta}\right)(x,0) dx.
 \end{align*}
Note that $\mathrm {supp} \;(w_{\epsilon, \eta, \alpha})=\mathrm {supp} \;(w_{\epsilon, \eta, \beta}) \subset \mc C_{\mc P}$ and $f(x), g(x)>0$ for $(x, 0) \in \mc C_{\mc P}$. 
Now
  based on 
  \begin{equation}\label{apbp}
  (a+b)^t\geq a^t+b^t+ta^{t-1}b+C_1ab^{t-1},\; a, b>0,\; t>2
  \end{equation}
we use the following inequality
 \[
 (m+n)^s(p+q)^t\geq m^{s}p^{t}+n^sq^t+sm^{s-1}np^t+tm^sp^{t-1}q+C_2mn^{s-1}q^t+C_3q^{t-1}n^sp, 
 \]
 for some non-negative constants $C_1, C_2, C_3$ and $p, q, m, n\geq 0,\; s, t\geq 2$ (see Appendix for the proof of \eqref{apbp}). Now using the above inequality together with Lemma \ref{3a} and  Young's inequality we get the following
\begin{align*}
&\mathcal I_{\lambda, \mu}(w_{0} + r\;w_{\epsilon,\eta, \alpha}, z_{0} + r\;w_{\epsilon,\eta, \beta})\leq \frac{a}{2}r^{2}(\|w_{\epsilon, \eta, \alpha}\|^2+\|w_{\epsilon, \eta, \beta}\|^2)\\&+ \frac{7}{4}{b}r^{4}(\|w_{\epsilon, \eta, \alpha}\|^{4}+\|w_{\epsilon, \eta, \beta}\|^4)-\frac{2}{{2^*_s}}r^{2^*_s}\int_{\Omega}|w_{\epsilon,\eta, \alpha}(x,0)|^{\alpha}|w_{\epsilon,\eta, \beta}(x,0)|^{\beta}dx\\&-{C_5r^{2^*_s-1}\int_{\Omega}|w_{\epsilon}(x,0)|^{2^*_s-1}dx}+ 2C_4b , 
\end{align*}
 where $C_4=(\|w_0\|^4+\|z_0\|^4)$ and $C_5$ is obtained on combining $C_2$ and $C_3$. Now assume
 \begin{align*}
 G(t)&=\frac{a}{2}t^{2}(\|w_{\epsilon, \eta, \alpha}\|^2+\|w_{\epsilon, \eta, \beta}\|^2)+ \frac74{b}t^{4}(\|w_{\epsilon, \eta, \alpha}\|^{4}+\|w_{\epsilon, \eta, \beta}\|^{4})\\&-\frac{2}{{2^*_s}}t^{2^*_s}\int_{\Omega}|w_{\epsilon,\eta, \alpha}(x,0)|^{\alpha}|w_{\epsilon,\eta, \beta}(x,0)|^{\beta}dx-C_5t^{2^*_s-1}\int_{\Omega}|w_{\epsilon}(x,0)|^{2^*_s-1}dx.
 \end{align*}
 \noi  Since $\displaystyle \lim_{t\rightarrow \infty} G(t)=-\infty$ and $\displaystyle \lim_{t\rightarrow 0^+} G(t)>0$.
 Therefore there exists $t_\epsilon>0$ such that
 \begin{align}\label{mm}
 G(t_\epsilon)=\displaystyle \sup_{t\geq 0}g(t)\;\textrm{and}\; \frac{d}{dt}G(t)\mid_{t=t_\epsilon}=0.
 \end{align}
 From \eqref{mm}, we get the following
\begin{equation}\label{der0}
 \begin{aligned}
 &{a}t_\epsilon (\|w_{\epsilon, \eta, \alpha}\|^2+\|w_{\epsilon, \eta, \beta}\|^2)+{7b}t_\epsilon^{3}(\|w_{\epsilon, \eta, \alpha}\|^{4}+\|w_{\epsilon, \eta, \beta}\|^{4})\\&=
2t_\epsilon ^{2^*_s-1}\ds \int_{\Omega}|w_{\epsilon,\eta, \alpha}(x,0)|^{\alpha}|w_{\epsilon,\eta, \beta}(x,0)|^{\beta}dx+C_6r^{2^*_s-2}\int_{\Omega}|w_{\epsilon}(x,0)|^{2^*_s-1}dx.
\end{aligned}
\end{equation}
 From \eqref{der0}, it is clear that $t_\epsilon$ is bounded from below. Moreover,
 \begin{equation}\label{bab}
  \begin{aligned}
& \frac{a}{t_\epsilon^2}(\|w_{\epsilon, \eta, \alpha}\|^2+\|w_{\epsilon, \eta, \beta}\|^{2})+7b (\|w_{\epsilon, \eta, \alpha}\|^{4}+\|w_{\epsilon, \eta, \beta}\|^{4})\\&=
    2{t_\epsilon^{2^*_s-4}}\ds \int_{\Omega}|w_{\epsilon,\eta, \alpha}(x,0)|^{\alpha}|w_{\epsilon,\eta, \beta}(x,0)|^{\beta}dx+C_6r^{2^*_s-5}\int_{\Omega}|w_{\epsilon}(x,0)|^{2^*_s-1}dx.
 \end{aligned}
 \end{equation}
 Since {$2^*_s\geq 4$}, $t_\epsilon$ is bounded above as well. If not then left hand side of \eqref{bab} is bounded for large values of $t_\epsilon$ , while right hand side goes to $+\infty$ as $t\rightarrow +\infty$, which a contradiction. Hence form the above arguments, there exist positive constants $t_1, t_2$, independent of $\lambda$ such that $0<t_1\leq t_\epsilon\leq t_2<\infty$. \\
 \noi Now from \cite{MR2911424, MR3117361} we get that the family $\{w_\epsilon\}$ and its trace on $\{y=0\}$ satisfy
\begin{equation}\label{estimates}
\begin{aligned}
  \|\eta w_\epsilon\|^2&=\|w_\epsilon\|^2+O(\epsilon^{n-2s})\\
  \displaystyle \int_{\Omega}|\eta w_\epsilon(x,0)|^{2^*_\alpha}dx&=\ds \int_{\mathbb R^N}\frac{1}{(1+|x|^2)^N} dx+O(\epsilon  ^N)\\
{\frac{\|\eta w_\epsilon\|^2}{\|\eta w_\epsilon\|^2_{2^*_s}}}&=\kappa_s S(s, N)+O(\epsilon^{N-2s})\\
\displaystyle \int_{\Omega}|\eta w_\epsilon(x,0)|^{2^*_s-1}&\geq C_7\epsilon^\frac{N-2s}{2}dx .
\end{aligned}
\end{equation}

Now using the estimates of \eqref{estimates}, we estimate second and fourth term in $G(t)$ as follows
\begin{align*}
&\frac74{b}t^{4}(\|w_{\epsilon, \eta, \alpha}\|^{4}+\|w_{\epsilon, \eta, \beta}\|^{4})-C_5t^{2^*_s-1}\int_{\Omega}|w_{\epsilon}(x,0)|^{2^*_s-1}dx\\
&\leq \frac74{b}t_2^{4}(\|w_{\epsilon, \eta, \alpha}\|^{4}+\|w_{\epsilon, \eta, \beta}\|^{4})-C_5t_1^{2^*_s-1}\int_{\Omega}|w_{\epsilon}(x,0)|^{2^*_s-1}dx\\
&\leq C_8b( S(s, N)^\frac{2N}{s}+O(\epsilon^{N-2s}))-C_9\epsilon^\frac{N-2s}{2}.
\end{align*}
Now we assume $b=\epsilon$ to get the following
\begin{align*}
&\frac74{b}t^{4}(\|w_{\epsilon, \eta, \alpha}\|^{4}+\|w_{\epsilon, \eta, \beta}\|^{4})-C_5t^{2^*_s-1}\int_{\Omega}|w_{\epsilon}(x,0)|^{2^*_s-1}dx\\&\leq  C_{10}\epsilon+C_{11}\epsilon^{N+1-2s}-C_9\epsilon^\frac{N-2s}{2}.
\end{align*}
Now using these estimates we get
 \begin{align*}
 &\displaystyle \sup_{t\geq 0}G(t)=G(t_\epsilon )\leq \displaystyle\left({a}(\alpha+\beta)t_\epsilon^{2}\|\eta w_{\epsilon}\|^2
 -\frac{2}{2^*_s}t_\epsilon^{2^*_s}\int_{\Omega}\alpha^\frac{\alpha}{2}{\beta}^\frac{\beta}{2} |\eta w_{\epsilon}(x,0)|^{2^*_s}dx\right)
 \\&\quad+C_{10}\epsilon+C_{11}\epsilon^{N+1-2s}-C_9\epsilon^\frac{N-2s}{2}\\
 &\leq \left(\frac{1}{2}-\frac{1}{2^*_s}\right)\left(\frac{1}{2}\right)^\frac{2}{2^*_s-2}(a\kappa_s )^\frac{2^*_s}{2^*_s-2}\left[\left(\frac{\alpha}{\beta}\right)^\frac{\beta}{\alpha+\beta}+
 \left(\frac{\beta}{\alpha}\right)^\frac{\alpha}{\alpha+\beta}\right]^\frac{2^*_s}{2^*_s-2}S(s, N)^\frac{2^*_s}{2^*_s-2}\\&\quad+O(\epsilon^{N-2s})+C_{10}\epsilon+C_{11}\epsilon^{N+1-2s}-C_9\epsilon^\frac{N-2s}{2},
  \end{align*}
where $C_9, C_{10}, C_{11}>0$ are positive constants independent of $\epsilon, \lambda$.
Now using Lemma \ref{optsys}, we get that
\begin{align*}
&\mathcal I_{\lambda, \mu}(w_{0} + r\;w_{\epsilon,\eta, \alpha}, z_{0} + r\;w_{\epsilon,\eta, \beta})
\leq \frac{2s}{N}\left(\frac12(a\kappa_s S(s, \alpha, \beta ))\right)^\frac{N}{2s}\\&+C_{12}\epsilon^{N-2s}+C_{10}\epsilon+C_{11}\epsilon^{N+1-2s}-C_9\epsilon^\frac{N-2s}{2}+2\epsilon R^4\\
&\leq \frac{2s}{N}\left(\frac12(a\kappa_s S(s, \alpha, \beta ))\right)^\frac{N}{2s}+C_{12}\epsilon^{N-2s}+C_{13}\epsilon+C_{11}\epsilon^{N+1-2s}-C_9\epsilon^\frac{N-2s}{2}.
\end{align*}
{Now if we choose $\epsilon$ in $w_{\epsilon}$ as $\epsilon^{1/p}$ for $p\geq N-2s$, then the above inequality will read as }
\begin{align*}
&\mathcal I_{\lambda, \mu}(w_{0} + r\;w_{\epsilon,\eta, \alpha}, z_{0} + r\;w_{\epsilon,\eta, \beta})\\
&\leq \frac{2s}{N}\left(\frac12(a\kappa_s S(s, \alpha, \beta ))\right)^\frac{N}{2s}+{C_{12}\epsilon^\frac{N-2s}{p}+C_{13}\epsilon+C_{11}\epsilon^\frac{N+p-2s}{p}
-C_9\epsilon^\frac{N-2s}{2p}}.
\end{align*}
Since $$\frac{N-2s}{p}=\ds \min\left\{\frac{N-2s}{p}, 1, \frac{N+p-2s}{p}\right\},$$
we get
\begin{align*}
&\mathcal I_{\lambda, \mu}(w_{0} + r\;w_{\epsilon,\eta, \alpha}, z_{0} + r\;w_{\epsilon,\eta, \beta})\\&\leq \frac{2s}{N}\left(\frac12(a\kappa_s S(s, \alpha, \beta ))\right)^\frac{N}{2s}+C_{14}\epsilon^\frac{N-2s}{p}-C_9\epsilon^\frac{N-2s}{2p}.
\end{align*}
Now observe that there exists $\epsilon_0>0$ sufficiently small  such that $C_{14}\epsilon^\frac{N-2s}{p}-C_9\epsilon^\frac{N-2s}{2p}>0$ for $\epsilon\in (0, \epsilon_0)$. Hence, for $\epsilon\in (0, \epsilon_0)$, there exist $\lambda, \mu>0$ satisfying the following inequality
\begin{align*}
&q(2-q)\left({(4-q)}\right)^\frac{2}{2-q}(a\kappa_sS(s, N))^\frac{q}{q-2}\left(\left(\lambda\|f\|_\gamma \right)^\frac{2}{2-q}+
\left(\mu\|g\|_\gamma\right)^\frac{2}{2-q}\right)
\\&\quad\quad\leq \epsilon^\frac{N-2s}{2p}(C_{14}\epsilon^\frac{N-2s}{2p}-C_9).
\end{align*}
Therefore, we can find a $\Gamma_*>0$ such that for $\epsilon\in (0, \epsilon_0)$ and $(\lambda, \mu)\in \mc M_{_{\Gamma_*}}$,
$\mathcal I_{\lambda, \mu}(w_{0} + r\;w_{\epsilon,\eta, \alpha}, z_{0} + r\;w_{\epsilon,\eta, \beta})\leq  c_{\lambda, \mu}$
 This proves the Proposition.
\end{proof}

 Now consider the following
\begin{eqnarray*}
  W_{1} &=& \left\{(w,z) \in \mc H(\mathcal C)\setminus\{0\} \big{|} \frac{1}{\|(w, z)\|}t^{-}\left(\frac{(w, z)}{\|(w, z)\|}\right) > 1\right\} \cup \{0\}, \\
   W_{2} &=& \left\{(w, z) \in \mc H(\mathcal C)\setminus\{0\} \big{|} \frac{1}{\|(w, z)\|}t^{-}\left(\frac{(w, z)}{\|(w, z)\|}\right) < 1\right\}.
\end{eqnarray*}
Then $\mc N_{\lambda, \mu}^{-}$ disconnects $ \mc H(\mathcal C)$ in two connected components $W_{1}$ and $W_{2}$
and $ \mc H(\mathcal C)\setminus \mc N_{\lambda, \mu}^{-} = W_{1} \cup W_{2}.$ For each $(w, z) \in \mc N_{\lambda, \mu}^{+},$ we have $1< t_{\max}((w, z)) < t^{-}((w, z)).$
 Since $t^{-}((w, z)) = \frac{1}{\|(w, z)\|}t^{-}\left(\left(\frac{(w, z)}{\|(w, z)\|}\right)\right),$ then $\mc N_{\lambda, \mu}^{+} \subset W_{1}.$
In particular, $(w_0, z_0) \in W_{1}.$  Now we claim that there exists $l_0>0$ such that {$(w_{0} + l_{0}w_{\epsilon, \eta}, z_0+l_0w_{\epsilon, \eta}) \in W_{2}. $}

 \noi First, we find a constant $c>0$ such that
 \[0 < t^{-} \left(\left(\frac{(w_{0}+l\;w_{\epsilon,\eta},z_{0}+l\;w_{\epsilon,\eta})}{\|(w_{0}+l\;w_{\epsilon,\eta}, z_{0}+l\;w_{\epsilon,\eta})\|}\right)\right)<c, {\quad \forall\, l>0}.\]
  Otherwise, there exists a sequence $\{l_{k}\}$ such that, as $k\rightarrow \infty$, $l_{k} \rightarrow \infty$ and {$t^{-}\left(\left(\frac{(w_{0}+l_k\;w_{\epsilon,\eta},z_{0}+l_k\;w_{\epsilon,\eta})}{\|(w_{0}+l_k\;w_{\epsilon,\eta}, z_{0}+l_k\;w_{\epsilon,\eta})\|}\right)\right) \rightarrow \infty$}. Let
  \[(\bar w_k, \bar z_k) = \frac{(w_{0}+l_k\;w_{\epsilon,\eta},z_{0}+l_k\;w_{\epsilon,\eta})}{\|(w_{0}+l_k\;w_{\epsilon,\eta}, z_{0}+l_k\;w_{\epsilon,\eta})\|}.\]
   Since $t^{-}((\bar w_k, \bar z_k))(\bar w_k, \bar z_k) \in \mc N_{\lambda, \mu}^{-} \subset \mc N_{\lambda, \mu}$ and by the Lebesgue dominated convergence theorem,
\begin{align*}
 \ds \lim_{k \rightarrow \infty}&\int_{\Omega}|\bar w_k(x,0)|^\alpha|\bar z_k(x, 0)|^\beta\\
 &=\ds \lim_{k \rightarrow \infty}\frac{\int_{\Omega} |(w_{0}+l_{k}\;w_{\epsilon,\eta})(x,0)|^{\alpha}(z_{0}+l_{k}\;w_{\epsilon,\eta})(x, 0)|^\beta dx}{\|(w_{0}+l_{k}\;w_{\epsilon,\eta},z_{0}+l_{k}\;w_{\epsilon,\eta} )\|^{2_s^*}} \\
 &=\ds \lim_{k \rightarrow \infty} \frac{\int_{\Omega} \vert ({w_{0}}/{l_{k}}+w_{\epsilon,\eta})(x,0)\vert ^{\alpha} \vert ({z_{0}}/{l_{k}}+w_{\epsilon,\eta})(x,0)\vert ^{\beta} dx}{\|({w_{0}}/{l_k}+\;w_{\epsilon,\eta},{z_{0}}/{l_k}+\;w_{\epsilon,\eta} )\|^{2_s^*}} \\
   &=\frac{ \int_{\Omega}(w_{\epsilon, \eta}(x,0))^{2_s^*}dx}{\|(w_{\epsilon, \eta},w_{\epsilon, \eta}) \|^{2_s^*}}.
\end{align*}
Now
\begin{align*}
  &\mathcal{I}_{\lambda, \mu}(t^{-}((\bar w_k, \bar z_k))(\bar w_k, \bar z_k)) = \frac{1}{2}a(t^{-}((\bar w_k, \bar z_k)))^2(\|\bar w_k\|^2+\|\bar z_k\|^2)\\&+\frac{1}{4}b(t^{-}(\bar w_k, \bar z_k))^4(\|\bar w_k\|^4+ \|\bar z_k)\|^4)-\frac{(t^{-}((\bar w_k, \bar z_k)))^{q}}{q}\lambda \int_{\Omega}f(x) |\bar w_k(x,0)|^q\;dx\\&-\frac{(t^{-}((\bar w_k, \bar z_k)))^{q}}{q}\mu \int_{\Omega}g(x) |\bar z_k(x,0)|^q\;dx \\&- \frac{2(t^{-}((\bar w_k, \bar z_k)))^{2^*_s}}{2^*_s}\int_{\Omega} |\bar w_{k}(x,0)|^{\alpha}|\bar z_{k}(x,0)|^{\beta}dx \rightarrow - \infty\;\; \textrm{as}\;\; k \rightarrow \infty,
\end{align*}
this contradicts that $\mc I_{\lambda, \mu}$ is bounded below on $\mc N_{\lambda, \mu}.$ \\
\noi Let
\begin{equation*}
    l_{0} = \frac{|c^{2}-\|(w_{0}, z_0)\|^{2}|^{\frac{1}{2}}}{\|(w_{\epsilon,\eta}, w_{\epsilon, \eta})\|} + 1,
\end{equation*}
then
\begin{align*}
  &\|(w_{0}+l_{0}w_{\epsilon,\eta}, z_{0}+l_{0}w_{\epsilon,\eta})\|^{2} = \|w_{0}\|^{2}+ \|z_{0}\|^{2} + 2(l_{0})^{2}\|w_{\epsilon,\eta}\|^{2}\\
  &+2l_{0}\langle w_{0}, w_{\epsilon,\eta}\rangle+ 2l_{0}\langle z_{0}, w_{\epsilon,\eta}\rangle>\|(w_{0}, z_0)\|^{2} + |c^{2}-\|(w_{0}, z_0)\|^{2}|\\& + 2l_{0}\langle w_{0}, w_{\epsilon,\eta}\rangle+2l_{0}\langle z_{0}, w_{\epsilon,\eta}\rangle>c^{2}>t^{-}\left(\left(\frac{(w_{0}+l_{0}w_{\epsilon,\eta},z_{0}+l_{0}w_{\epsilon,\eta} )}{\|(w_{0}+l_{0}w_{\epsilon,\eta}, z_{0}+l_{0}w_{\epsilon,\eta})\|}\right)\right)^{2}
\end{align*}
that is $(w_{0}+l_{0}w_{\epsilon,\eta}, z_{0}+l_{0}w_{\epsilon,\eta}) \in W_{2}.$\\
\noindent \textbf{Proof of Theorem \ref{22mht1} (ii):}
Let us fix $\Gamma_{00}=\ds \min\{\Gamma_0, \Gamma_3, \Gamma_*\}$
and define a path connecting $W_1$ and $W_2$ as $\gamma_0(t) = (w_0+t\;l_{0}w_{\epsilon,\eta}, z_0+tl_{0}w_{\epsilon,\eta})$ for $t \in [0,1],$ then there exists $t_{0} \in (0,1)$ such that $(w_0+t_0\;l_{0}w_{\epsilon,\eta}, z_0+t_0l_{0}w_{\epsilon,\eta}) \in \mc N_{\lambda, \mu}^{-}$. Therefore, by Lemma \ref{II},
    $$
    \Theta_{\lambda, \mu}^- \leq \mc I_{\lambda, \mu}(w_0+t_0\;l_{0}w_{\epsilon,\eta}, z_0+t_0l_{0}w_{\epsilon,\eta})< c_{\lambda,\mu}
    $$
for $0<\Gamma<\Gamma_{00}$.
Now from Proposition \ref{prp1} $(ii)$, there exists a Palais-Smale sequence
$\{(w_{k}, z_k)\} \subset \mc N_{\lambda, \mu}^{-}$ such that $\mc I_{\lambda, \mu}((w_{k}, z_k))\rightarrow \Theta_{\lambda, \mu}^-$. Since $\Theta_{\lambda, \mu}^{-} <c_{\lambda, \mu}$, by Proposition \ref{crcmp}, upto a subsequence there exists  $(w^0, z^0)$ in  $\mc H(\mathcal C)$ such that $(w_{k}, z_k) \rightarrow (w^0, z^0)$ strongly in $ \mc H (\mathcal C).$ Now using Corollary \ref{nlclosed}, $(w^0, z^0) \in \mc N_{\lambda, \mu}^{-}$ and $ \mc I_{\lambda, \mu}(w^0, z^0) = \Theta_{\lambda, \mu}^{-}.$ Therefore $(w^0, z^0)$ is also a solution. Moreover, $\mc I_{\lambda, \mu}(w, z)=\mc I_{\lambda, \mu}(|w|, |z|)$, we may assume that $w^0\geq 0, z^0\geq 0$. Again using the similar argument as in case of first solution and strong maximum principle (see Lemma 2.6 of \cite{MR2825595}) we conclude that  $w^0>0, z^0>0$ is a second positive solution
 of the problem $(S_{\lambda, \mu})$. Since $\mc N_{\lambda, \mu} ^+\cap \mc N_{\lambda, \mu}^-=\emptyset$, $(w_0, z_0)$ and $(w^0, z^0)$ are distinct. This proves Theorem \ref{22mht1}.
\section*{Appendix}
\textbf{Proof of inequality \eqref{apbp}}:
To prove \eqref{apbp}, it is enough  to show that 
\begin{equation}\label{6.11}
(1+m)^p\geq 1+m^p+pm+C_1m^{p-1}
\end{equation}
for $m>0$. Equivalently, it is enough (divide \eqref{6.11} by $m^p$ and put $t=1/m$) to prove the following inequality to show \eqref{6.11} 
\begin{equation}\label{6.12}
(1+t)^p\geq t^p+1+pt^{p-1}+C_1t
\end{equation} 
for $t>0$. Let us introduce $\mathcal G:(0, \infty)\to \mathbb R$ as 
\[
\mathcal G(t)=\frac{(1+t)^p-t^p-1-pt^{p-1}}{t}
\] 
Then it is easy to show that $\mathcal O(t)=(1+t)^p-t^p-1-pt^{p-1}> 0$ for $t>0$ and $p>2$ (one can prove by showing the monotonicity of $\mathcal O(t)$ and observing $\mathcal O(0)=0$).
Now we have the following behavior of $\mathcal G(t)$
\begin{equation}\label{6.13}
\displaystyle \lim_{t\to 0^+}\mathcal G(t)=p\;\; \text{ and } \displaystyle \lim_{t\to +\infty}\mathcal G(t)=+\infty.
\end{equation}
From \eqref{6.13}, using the continuity of $\mathcal G(t)$, there exists some $C_1>0$ ( may depend on $p$) such that $\mathcal G(t)\geq C_1$ which proves \eqref{6.12} and consequently \eqref{apbp}.

\end{document}